\documentclass{jgcc}

\pdfoutput=1
\usepackage{lastpage}
\jgccdoi{12}{1}{1}
\jgccheading{}{\pageref{LastPage}}{}{}%
{Sept.~11,~2019}{Feb.~29,~2020}{}


\RequirePackage{myPackages}

\RequirePackage{myMacros}
\RequirePackage{myTikz}

\begin{document}

\title{On the lattice of subgroups of a free group: complements and rank}

\author[J.~Delgado]{Jordi Delgado}	
\address{Centro de Matemática, Universidade do Porto, Portugal}	
\email{jdelgado@crm.cat}  

\author[P.~V.~Silva]{Pedro~V.~Silva}	
\address{Centro de Matemática, Universidade do Porto, Portugal}	
\email{pvsilva@fc.up.pt}  






\begin{abstract}
  A $\join$-complement of a subgroup $H \leqslant \Fn$ is a subgroup $K \leqslant \Fn$ such that $H \join K = \Fn$.
    If we also ask $K$ to have trivial intersection with $H$, then we say that $K$ is a \mbox{$\oplus$-complement} of $H$.
    The minimum possible rank of a $\join$-complement (\resp $\oplus$-complement) of $H$ is called the $\join$-corank (\resp $\oplus$-corank) of $H$.
    We use Stallings automata to study these notions and the relations between them. In particular, we characterize when complements exist, compute the $\join$-corank, and provide language-theoretical descriptions of the sets of cyclic complements.
    Finally, we prove that the two notions of corank coincide on subgroups that admit cyclic complements of both kinds.
\end{abstract}

\maketitle

\section{Introduction}

Subgroups of free groups are complicated. Of course not the structure of the subgroups themselves (which are always free, a classic result by Nielsen and Schreier) but the relations between them, or more precisely, the lattice they constitute. A first hint in this direction is the fact that (free) subgroups of any countable rank appear as subgroups of the free group of rank $2$
(and hence of any of its noncyclic subgroups) giving rise to a self-similar structure.

This scenario quickly provides challenging questions involving ranks. Classical examples include intersections
of finitely generated subgroups, and subgroups of fixed points of automorphisms; both proved to be finitely generated in the second half of last century (see~\cite{HowsonIntersectionFinitelyGenerated1954} and
\cite{GerstenFixedPointsAutomorphisms1987} respectively),
and both having a long and rich subsequent history in the quest for bounds for those finite ranks (see
\cite{FriedmanSheavesGraphsTheir2015,
MineyevSubmultiplicativityHannaNeumann2012}
and~\cite{BestvinaTrainTracksAutomorphisms1992} repectively).


In this paper we shall be mainly concerned with a kind of dual of the previous problem: given a finitely generated subgroup $H$ of $\Fn$, what is the minimum number of generators that must be added to $H$ in order to generate the full group~$\Fn$?
What happens if the added subgroup is also required to intersect trivially with $H$?
These numbers, called respectively the \emph{join} ($\join$) and \defin{direct ($\oplus$) coranks} of $H$, shall be investigated using Stallings automata.

Although previously studied using other techniques, maybe the most enlightening approach to subgroups of the free group $\Fn$ was their geometric interpretation as covering spaces of the bouquet of $n$ circles.
It soon became clear that the (mainly topological) original viewpoint by Serre and Stallings
(see~\cite{SerreTrees1980,
StallingsTopologyFiniteGraphs1983})
admitted an appealing restatement in terms of automata
(see~\cite{KapovichStallingsFoldingsSubgroups2002, BartholdiRationalSubsetsGroups2010}).
We briefly summarize this modern approach
in Section~\ref{sec: Stallings automata}.
Furthermore, similarly to the strategy followed in \cite{SilvaAlgorithmDecideWhether2008}
and
\cite{PuderPrimitiveWordsFree2014}
to deal with free factors, we highlight the role played by identifications of vertices in the Stallings automaton.

In Section~\ref{sec: complements and coranks}, we introduce the various notions of complement and corank studied in this paper.

Section~\ref{sec: join} is devoted to join complements and join corank. We discuss
the possible combinations between rank and join corank, show that the latter is always computable
(a result previously proved by Puder in \cite{PuderPrimitiveWordsFree2014}),
and prove that the set of join cocycles is always rational. We note that, from a language theoretical viewpoint, proving that a set of solutions is rational is highly desirable: in view of all the closure properties satisfied by rational languages, this allows for an efficient search for solutions of particular types.

In Section~\ref{sec: meet}, the simpler case of
meet complements
is discussed. Existence is essentially determined by the subgroup having finite or infinite index, and this will be useful for the discussion of direct complements.

Section~\ref{sec: direct} contains the main results of the paper, devoted to the harder case of direct complements and coranks. Existence is shown to depend on the index of the subgroup only, and the possible combinations between rank and direct corank turn out to be the same as in the join case. The set of cyclically reduced direct cocycles is also rational, but the set of direct cocycles needs not to be: in general, it is only context-free (the next level above rational, in the classical Chomsky's hierarchy from language theory). We also prove that the concepts of join cocyclic and direct cocyclic coincide, and raise the natural question: do join corank and direct corank coincide in the general case?

We finally point out that Stallings' geometric interpretation converts the corank problems into problems about equations in automata
(that, is equations between automata that include arcs labelled by variable strings). We note that this is a very appealing general problem that, in particular, includes that of equations in the free group which has become one of the main topics in modern group theory
(see~\cite{
MakaninEquationsFreeGroups1982,
RazborovSystemsEquationsFree1984,
JezRecompressionSimplePowerful2013}).

\section{Preliminaries}

Throughout the paper we assume that $A = \set{a_1,\ldots,a_n}$ is a finite set of \defin{letters} that we call an \defin{alphabet},
and we denote by $A^{\!*}$ the \defin{free monoid} on $A$ (consisting of all finite words on $A$ including the \defin{empty word},
which is denoted by $\emptyword$).

The subsets of $A^{\!*}$ are called languages (over $A$), or $A$-languages.
An $A$-language is said to be \defin{rational} if it
can be  obtained from the letters of $A$ using the operators union, product and star (submonoid generated by a language) finitely many times. A direct consequence of a fundamental theorem of Kleene (see for example \cite[Theorem 2.1]{SakarovitchElementsAutomataTheory2009})
is that the set of rational $A$-languages is closed under finite intersection and complement.

We also denote by $\Fn$ the \defin{free group} with basis $A$; that is, $\Fn = \Free[A] =  \pres{A}{-}$.
More precisely, we denote by~$A^{-1}$ the \defin{set of formal inverses} of~$A$. Formally, $A^{-1}$ can be defined as a set $A'$ equipotent and disjoint with $A$, together with a bijection $\varphi \colon A \to A'$; then, for every $a \in A$, we call $a \varphi$ the \defin{formal inverse} of $a$, and we write~$a^{-1} = a \varphi$. So $A^{-1} = \{ a^{-1} \st a \in A \}$, and $A \cap A^{-1} = \varnothing$.

The set $A^{\pm} = A \sqcup A^{-1}$,
called the \defin{involutive closure} of $A$,
is then equipped with an involution $^{-1} 
$ (where $(a^{-1})^{-1} = a$), which can be extended to
$(A^{\pm})^*$
in the natural way: $(a_1\ldots a_n)^{-1} = a_n^{-1}\ldots a_1^{-1}$ for all $a_1, \ldots,a_n \in A^{\pm}$. An alphabet is called \defin{involutive} if it is the involutive closure of some other alphabet.

A word in $(A^{\pm})^{*}$ is said to be \defin{(freely) reduced} if it contains no consecutive mutually inverse letters (\ie it has no factor of the form $a a^{-1}$,
where $a \in A^{\pm}$).
It is well known that
the word obtained from $w \in (A^{\pm})^{*}$ by successively removing pairs of consecutive inverse letters is unique; we call it the \defin{free reduction} of $w$, and we denote it by $\red{w}$.
Similarly, we write $\red{S} = \set{\red{w} \st w \in S}$, for any subset $S \subseteq (A^{\pm})^{*}$,
and we denote by $\Red_A$ (or $\Red_n$) the set of reduced words in $(A^{\pm})^{*}$, that is the rational language
\begin{equation*}
\Red_A
\,=\,
\red{(A^{\pm})^{*}}
\,=\,
(A^{\pm})^{*}
\setmin
\textstyle{
\bigcup_{a \in A^{\pm}} (A^{\pm})^{*} a a^{-1} (A^{\pm})^{*}.
}
\end{equation*}
In a similar vein, a word in $A^{\pm}$ is said to be \defin{cyclically reduced} if all of its cyclic permutations are reduced
(that is, if it is reduced and its first and last letters are not inverses of each other). The \defin{cyclic reduction} of a word $w \in \Fn$,
denoted by $\cred{w}$,
is obtained after iteratively removing from $\red{w}$ the first and last letters if they are inverses of each other.
We also extend this notation to subsets, and denote by $\Cred_A$ (or $\Cred_n$) the set of all cyclically reduced words in $(A^{\pm})^{*}$, that is the rational language
\begin{equation*}
\Cred_A
\,=\,
\cred{(A^{\pm})^{*}}
\,=\,
(A^{\pm})^{*}
\setmin
\textstyle{
\bigcup_{a \in A^{\pm}}
\left(
(A^{\pm})^{*} a a^{-1} (A^{\pm})^{*}
\cup
a (A^{\pm})^{*} a^{-1}
\right)}.
\end{equation*}
The free group $\Free[A]$ (with basis $A$) can then be thought as the set
of reduced words in $A^{\pm}$ with the operation consisting of ``concatenation followed by reduction''. We recall that Benois' Theorem (see~\cite{BenoisPartiesRationnellesGroupe1969}) allows us to understand the rational subsets of $\Free[A]$ as reductions of rational $A^{\pm}$-languages.

\begin{defi}
    The \defin{rank of a group} $G$, denoted by~$\rk (G)$, is the smallest cardinality of a generating set for $G$.
\end{defi}

It is well known that the free group $\Fn$ has rank $n$, and that every subgroup of $\Fn$ is again free and can have any countable rank if $n\geq 2$.
We will see in Theorem~\ref{thm: Stallings bijection} that one can biunivocally assign to every subgroup $H \leqslant \Fn$ a geometric object
--- called the \emph{Stallings automaton} $\stallings{H}$ of $H$ --- which provides a lot of useful information about the subgroup. What is more, if $H$ is given by a finite family of generators, then $\stallings{H}$ is fastly computable (see~\cite{TouikanFastAlgorithmStallings2006}), and many algorithmic results regarding subgroups of the free group follow smoothly from the Stallings construction.

For example,
we shall see that
the rank of a subgroup $H$ of $\Fn$ is precisely the (graph) rank of $\stallings{H}$; hence one can always compute the rank of a finitely generated subgroup $H$ of $\Fn$ from a finite family of generators.

One of the goals in this paper is to to obtain an analogous result for the join corank of $H$, that is, the minimum number of elements that must be added to $H$ in order to generate the whole group~$\Fn$
(see Theorem~\ref{thm: join-corank is computable}).

\subsection{Graphs, digraphs, and automata}

As we have already mentioned, we will use automata (that is essentially labelled digraphs) to describe subgroups of the free group. Since parallel arcs and loops are allowed, we shall define digraphs in the sense of Serre.

\begin{defi} \label{def: digraph}
  A \defin{directed multigraph} (or \defin{digraph} for short) is a tuple $\dGri = (\Verts,\Edgi, \init, \term$), where
     $\Verts$ is a nonempty set
     (called the set of \defin{vertices} of $\dGri$),
     $\Edgi$ is a set
     (called the set of \defin{arcs} or \defin{directed edges} of $\dGri$), and
     $\init,\term \colon \Edgi \to \Verts$ are (resp.~initial and final) \defin{incidence functions}.
Then, for each arc $\edgi \in \Edgi$, we say that $\edgi$ is \defin{incident} to $\init(\edgi)$ and $\term(\edgi)$, which are called the \defin{origin} (or \defin{initial vertex}), and \defin{end} (or \defin{final vertex}) of $\edgi$, respectively. Two vertices are said to be \defin{adjacent} if there exists an edge incident to both of them, and
two edges are said to be
\defin{incident} if some vertex is incident to both of them.

We denote by $\Verts \dGri$ and $\Edgi \dGri$ the respective sets of vertices and arcs of a digraph~$\dGri$.
A digraph $\dGri
$ is called \defin{finite}
if the cardinal
$\card{(\Verts \dGri \sqcup \Edgi \dGri)}$
is finite.

Note that no incidence restrictions have been applied; in particular
we are allowing both the possibility of arcs having the same vertex as origin and end (called \defin{directed loops}), and of different arcs sharing the same origin and end (called \defin{parallel arcs}).
\end{defi}

\begin{defi}
   A \defin{walk} in a digraph $\dGri$ is a finite alternating sequence
   $\walki = \verti_0 \edgi_1 \verti_1 \ldots \edgi_{n} \verti_{n}$ of successively incident
   vertices and arcs
   in $\dGri$
   (more precisely
   $\init{\edgi_i} = \verti_{i-1}$ and $\term{\edgi_i} =
   \verti_{i}$ for $i = 1,\dots,n$).
   We then say that $\walki$ goes from $\verti_0$ to $\verti_n$
   --- or that $\walki$ is a \defin{$(\verti_0,\verti_n)$-walk} --- and we write $\walki \colon \verti_0 \xleadsto{} \verti_n$. If the first and last vertices of $\walki$ coincide then we say that $\walki$ is a \defin{closed walk}.
   A closed walk from $\verti$ to $\verti$
   is called a \defin{$\verti$-walk}.
   The \defin{length of a walk} is the number of arcs in the sequence. The walks of length $0$, called \defin{trivial walks}, correspond precisely to the vertices in $\dGri$.
\end{defi}

\begin{defi} \label{def: pointed automaton}
  Given an alphabet $A$,
  a
  \defin{(pointed)
  $A$-automaton}
  $\Ati = (\Verti,\Edgi, \init, \term, \lab, \bp)$  
  is a
  digraph $\dGri = (\Verti,\Edgi, \init, \term)$
  called the \defin{underlying digraph}
  of $\Ati$,
  together with a labelling on the arcs $\lab\colon \Edgi \to A$, and a distinguished vertex~\bp\ called the \defin{base vertex} or \defin{basepoint} of $\Ati$.
\end{defi}

\begin{rem}
    In the general automata setting, pointed automata correspond to automata having a unique \emph{coincident} initial and final state.
\end{rem}

A vertex $\verti$ in an $A$-automaton is said to be \defin{complete} if for every letter $a \in A$ there is (at least) one $a$-arc with origin $\verti$.
An automaton is \defin{complete} if all its vertices are complete.
Otherwise
(if there exists a vertex $\verti$ and a letter $a$ such that there is no $a$-arc with origin $\verti$), we say that both the vertex $\verti$ and the automaton are $a$-\defin{deficient}.
A vertex is said to be an
\defin{$a$-source}
if it is the origin of one single arc, and this arc is labelled by $a$. An automaton whose basepoint is an $a$-source is also called an $a$-source.

An $A$-automaton
is said to be \defin{deterministic}
if no two arcs with the same label depart from the same vertex.
Hence, a deterministic automaton is complete if and only if
every letter induces a total transformation on the vertex set.

\begin{defi} \label{def: label of a walk}
The \defin{label of a walk} $\walki$ in an $A$-automaton $\Ati
= (\Verts,\Edgi, \init, \term, \lab, \bp)
$ is defined to be
\begin{equation*}
  \lab(\walki) =
  \left\{\!
  \begin{array}{lll}
       \emptyword \, , & \text{if } \walki = \verti \, , &\text{(an empty walk  in $\Ati$)},\\
       \lab(\edgi_1) \cdots \lab(\edgi_k) \, , & \text{if } \walki = \verti_0 \edgi_1 \verti_1 \ldots \edgi_k \verti_k &\text{(a nonempty walk  in $\Ati$)\,.}
  \end{array}\right.
\end{equation*}
We then say that $\walki$ \defin{reads} or \defin{spells} the word $\lab(\walki) \in A^{*}$, and that the word $\lab(\walki)$ \defin{labels} the walk $\walki$.
We also write
$\lab(\Ati) = \set{\lab(\edgi) \st \edgi \in \Edgi} \subseteq A$
(the subset of letters appearing as labels of arcs in $\Ati$).
If $\Verti$ and $\Vertii$ are subsets of vertices in a labelled digraph $\Ati$, then we denote by $\lang[\Verti \Vertii]{\Ati}$
the set of words read by walks from vertices in $\Verti$ to vertices in $\Vertii$.
In particular, if $\verti,\vertii$ are vertices of $\Ati$, then
$\lang[\verti \vertii]{\Ati} = \set{\,\lab(\walki) \st \walki \colon \verti \xleadsto{\ } \vertii \text{ in } \Ati \,}$.
In view of the well-known Kleene's Theorem the languages $\lang[\Verti\Vertii]{\Ati}$ are rational if $\Ati$ is finite.

The set of words read by $\bp$-walks in an $A$-automaton $\Ati$ is called the \defin{language recognized} by $\Ati$, and is denoted by $\lang{\Ati}$; that is $\lang{\Ati} = \lang[\bp\bp]{\Ati}$. It is clear that $\lang{\Ati}$ is a submonoid of the free monoid $A^{\!*}$.
\end{defi}

\begin{defi} \label{def: involutive automaton}
 An \defin{involutive $A$-automaton}
 is an
  $A^{\pm}$-automaton
  together with an involution
  $\edgi \mapsto \edgi^{-1}$ on its arcs (called \defin{inversion of arcs}) such that:
  \begin{enumerate}[ind]
    \item No arc is the inverse of itself (\ie $\edgi^{-1} \neq \edgi$, for every $\edgi \in \Edgi$).
    \item Inverse arcs are reversed (\ie $\init{\edgi^{-1}} = \term{\edgi}$, for every $\edgi \in \Edgi$).
    \item Arc inversion is compatible with label inversion (\ie $\lab(\edgi^{-1}) = \lab(\edgi)^{-1}$, for every~${\edgi \in \Edgi}$).
  \end{enumerate}
\end{defi}
Thus, in an involutive automaton, for every labelled arc $\edgi \equiv \verti \xarc{\ } \vertii $ reading $a \in A$, there always exists a  reversed arc
  $\edgi^{-1} \equiv \verti \xcra{\ } \vertii$ reading $a^{-1}$ (called the \defin{inverse} of $\edgi$), so that arcs appear by pairs as shown in Figure~\ref{fig: involutive arc}.

\begin{figure}[!htb]
\centering
\begin{tikzpicture}[shorten >=3pt, node distance=.3cm and 2cm, on grid,auto,>=stealth']
   \node[state] (0) {};
   \node[state] (1) [right = of 0] {};
   \path[->]
        ([yshift=0.3ex]0.east) edge[]
            node[pos=0.5,above=-.2mm] {$a$}
            ([yshift=0.3ex]1)
   ([yshift=-0.3ex]1.west) edge[dashed]
            node[pos=0.5,below=-.2mm] {$a^{-1}$}
            ([yshift=-0.3ex]0);
\end{tikzpicture}
\vspace{-5pt}
\caption{Arcs in an involutive automaton}
\label{fig: involutive arc}
\end{figure}

A walk in an involutive automaton is said to be \defin{reduced} if it has no two consecutive inverse arcs.

An arc in an involutive $A$-automaton is said to be \defin{positive} (resp.\ \defin{negative})
if it is labelled with a letter in~$A$ (resp.~$A^{-1}$).
We respectively denote by~$\Edgi^{+} \Ati$ and $\Edgi^{-} \Ati$
the sets of positive
and negative
arcs in an involutive automaton~$\Ati$.
The \defin{positive representation} of an $A$-involutive automaton $\Ati$ is the $A$-automaton obtained after removing all the negative arcs from $\Ati$.

\begin{rem}
  An involutive automaton is fully characterized by its positive representation
  (with the tacit assumption that every positive arc,
  say reading $a\in A$,
  is allowed to be crossed backwards reading the inverse label $a^{-1}$).
\end{rem}

  \begin{defi}
  The \defin{
  underlying graph}
  of an involutive automaton $\Ati$ is the undirected multigraph, denoted by $\Gamma$\,, obtained by
  identifying all the pairs of respectively inverse arcs in the underlying digraph of $\Ati$. Note that this is the same as `forgetting the labels and direction' in the positive 
  representation of $\Ati$.
\end{defi}

\begin{rem}
  Any undirected multigraph $\Gri$ can be seen as the underlying graph of some involutive automaton; an edge in $\Gamma$ is then an unordered pair $\set{\edgi,\edgi^{-1}}$.
  We shall refer to undirected multigraphs simply as graphs (in contraposition to digraphs, introduced in Definition~\ref{def: digraph}).
\end{rem}

\begin{conv}
If not stated otherwise, the automata appearing in this paper hereinafter will be assumed to be pointed and involutive,
and we shall represent them by their positive part.
It is clear that the language recognized by such an automaton is the same as the language recognized by the connected component containing the basepoint.
With this convention, an (involutive) $A$-automaton $\Ati$ is complete if and only if (in the positive representation of~$\Ati$), for every vertex $\verti$ and every positive letter $a \in A$, there is an $a$-arc starting at $\verti$ and an $a$-arc ending at $\verti$.
\end{conv}

 \begin{defi} \label{def: graph rank}
The \defin{(cycle) rank} of a
graph $\Gri$, denoted by~$\rk(\Gri)$, is the minimum number of edges that must be removed from $\Gri$ to obtain a forest (\ie to break all the cycles in~$\Gri$).
\end{defi}
It is well known (see for example~\cite{BergeTheoryGraphs2001}) that, if $\Gri$ is finite, then
\begin{equation} \label{eq: graph rank}
    \rk (\Gri)
    \,=\,
    e - v + c \,,
\end{equation}
where $e$,$v$, and $c$ are respectively the number of edges, vertices, and connected components in $\Gri$.

The previous considerations make it possible (and convenient) to extrapolate graph-theoretical notions to involutive automata from their underlying graphs.
For example, the \defin{rank}, \defin{connectivity}, or \defin{vertex degree} of an automaton $\Ati$ are defined in terms of the homonymous notions in its underlying graph.
In the same vein, we will call an involutive automaton a \defin{path}, a \defin{cycle}, a \defin{tree}, or a \defin{spanning tree} if their underlying graph is~so.

\begin{rem}
If an (involutive) automaton $\Ati$ is deterministic, then a walk $\walki$ in $\Ati$ is reduced if and only if its
label $\ell(\walki)$ is reduced.
\end{rem}

\begin{defi}
The \defin{reduced label} of a walk $\walki$ in an involutive $A$-automaton is
defined to be
$\red{\lab}(\walki) = \red{\lab(\walki)} \in \Free[A]$.
We write
$\rlang[\Verti \Vertii]{\Ati} =
\red{\lang[\Verti \Vertii]{\Ati}}$.
The set of reduced labels of $\bp$-walks in $\Ati$ is a subgroup of $\Free[A]$ called the \defin{subgroup recognized} by $\Ati$, and is denoted by $\gen{\Ati}$, so that
$\gen{\Ati} = \rlang{\Ati} \leqslant \Free[A]$.
\end{defi}

\begin{defi}
    An
    automaton $\Ati$ is said to be  \defin{core} if every vertex appears in some \mbox{$\bp$-walk} with reduced label.
\end{defi}
    Note that, if $\Ati$ is deterministic, this is the same as $\Ati$ being connected and not having ``hanging trees'' not containing the basepoint (we speak of {\em hanging trees} when we have a tree adjoined to a vertex of degree strictly greater than $1$ in a graph). Accordingly, we define the \defin{core} of a deterministic automaton~$\Ati$, denoted by~$\core(\Ati)$, to be the
    automaton obtained after taking the basepoint component of $\Ati$ and removing from it all the hanging trees not containing the basepoint. Note that we then have that $\gen{\core (\Ati)} = \gen{\Ati}$.

\begin{defi}
    An
    involutive pointed
    automaton is said to be \defin{reduced} if it is deterministic and core (and hence connected).
\end{defi}

\begin{defi}
A \defin{subautomaton} of an
automaton $\Ati$ is any
automaton
obtained from $\Ati$ through restriction
(of vertices or arcs).
We then write $\Atii \leqslant \Ati$
(or $\Atii < \Ati$, if $\Atii \neq \Ati$,
that is if $\Atii$ is a \defin{strict subautomaton} of $\Ati$).
\end{defi}

\subsection{Operations on automata}

Throughout this paper, two main kinds of operations on automata appear, namely vertex and arc identification.

\begin{defi}
Let $\Ati$ be an $A$-automaton
and let $\Verti$ be a nonempty subset of vertices of~$\Ati$.
The
\defin{quotient of $\Ati$ by $\Verti$},
denoted by $\qt{\Ati}{\Verti}$,
is
defined to be
the automaton obtained after identifying
in $\Ati$
the vertices in $\Verti$
(and inheriting the adjacencies,
the labelling,
and the initial and terminal vertices from $\Ati$).
\end{defi}

\begin{rem}
Note that
$\card {\Verts(\qt{\Ati}{\Verti})} = \card {\Verts\Ati} - {\card \Verti} + 1$,
and
$\card {\Edgs (\qt{\Ati}{\Verti})} = \card {\Edgs\Ati}$. Hence,
\begin{equation} \label{eq: rank graph quotient}
    \rk \Ati
    \,\leq\,
    \rk \qt{\Ati}{\Verti}
    \,\leq\,
    \rk \Ati + \card {\Verti} - 1 \,,
\end{equation}
where the lower (\resp upper) bound in \eqref{eq: rank graph quotient} corresponds to the case where all the vertices in $\Verti$ belong to different (\resp the same) connected components of~$\Ati$. In particular, identification of vertices can never decrease the rank of an automaton.

If $\Verti$ is finite, the exact rank of $\qt{\Ati}{\Verti}$ is immediately obtained from the case of two vertices
(denoted by
$\qt{\Ati}{\verti \shorteq \vertii}
=
\qt{\Ati}{\set{\verti,\vertii}}$),
which is detailed below.
\end{rem}

\begin{lem}
Let $\verti,\vertii$ be two different vertices in $\Ati$, then:
\begin{equation*} \label{eq: rk Ati + w}
    \rk \qt{\Ati}{\verti \shorteq \vertii}
    \,=\,
    \left\{ \!
    \begin{array}{ll}
            \rk \Ati +1 &  \text{if $\verti$ and $\vertii$ are connected in $\Ati$,} \\
            \rk \Ati  &  \text{if $\verti$ and $\vertii$ are not connected in $\Ati$.}
    \end{array}
    \right.
\end{equation*}
\end{lem}

Hence, if $\Ati$ is finite, the identification of two vertices either increases the rank exactly by one (if they are connected) or keeps the rank equal (if they are disconnected).

\begin{defi}
 Let $\Ati$ and $\Atii$ be $A$-automata. The
 \defin{sum}
 of $\Ati$ and $\Atii$, denoted by $\Ati + \Atii$,
 is then the automaton obtained after identifying the basepoints of $\Ati$ and $\Atii$.
\end{defi}

It is clear that the sum of automata is associative and commutative.
We write $\sum_{i=1}^{n} \Ati_{\!i} = \Ati_{\!1} + \ldots + \Ati_{\!n} $.
Then,
it is also clear that
$\rk (\sum_i \Ati_{\!i}) = \sum_i \rk(\Ati_{\!i})$,
and that
$\Gen{\sum_i \Ati_{\!i}} =
\Gen{\bigcup_i \,\gen{\Ati_{\!i}}} =
\bigvee_{\!i} \gen{\Ati_{\!i}}$.

Regarding arcs, we shall only be interested in a very specific kind of identification.
\begin{defi}
A \defin{folding} in an automaton
$\Ati$
is the identification of two different arcs
in $\Ati$
with the same origin and label
(inheriting the adjacencies,
the labelling,
and the basepoint from $\Ati$).
A folding is said to be
\defin{closed} if the identified arcs are parallel
and \defin{open} otherwise.
\end{defi}

Recall that our automata are always involutive and we use the convention of only representing its positive part,
that is, we are always implicitly assuming that every time a folding is performed (in the positive part) the corresponding folding between the respectively inverse arcs is also performed. Hence every folding (of arcs) induces an identification between the corresponding incident edges in the underlying graph.

\begin{rem} \label{rem: effect of foldings}
Note that, since foldings do not change the number of connected components in an automaton, the following statements are equivalent:
\begin{enumerate}[dep]
    \item a folding is closed (\resp open);
    \item a folding does not produce (\resp produces) an
    identification of vertices;
    \item a folding reduces by one (\resp does not change) the rank of the automaton.
\end{enumerate}
Therefore,
if $\widetilde{\Ati}$ is obtained from $\Ati$ after a sequence of foldings
(\ie if $\widetilde{\Ati}$ is a \defin{reduction} of~$\Ati$),
then
\begin{enumerate}[ind]
    \item $\card {\Edgs \widetilde{\Ati}} < \card {\Edgs \Ati}$
    \quad
    (precisely, $\card {\Edgs^+ \widetilde{\Ati}} = \card {\Edgs^+ \Ati} - \card{ \text{foldings}}$);
    \item $\card \Verts \widetilde{\Ati} \leq \card \Verts \Ati$
    \quad
    (precisely, $\card {\Verts \widetilde{\Ati}} = \card {\Verts \Ati} - \card {\text{open foldings}}$);
    \item $\rk \widetilde{\Ati} \leq \rk \Ati$
    \quad
    (precisely, $\rk \widetilde{\Ati} = \rk \Ati - \card {\text{closed foldings}}$);
    \item $\gen{\widetilde{\Ati}} = \gen{\Ati}$.
   \end{enumerate}
\end{rem}

\section{Subgroups of free groups and Stallings automata} \label{sec: Stallings automata}

We have seen that one can naturally assign a subgroup of $\Fn =\pres{A}{-}$
to every $A$-automaton $\Ati$. We shall see that every subgroup of $\Fn$ admits such a description,
which
can be made unique after adding some natural restrictions to the used automata.


\begin{nota}
If $\verti$ and $\vertii$ are vertices in a tree $T \leqslant \Ati$,
then we denote  by $\verti \xleadsto{\scriptscriptstyle{T}} \vertii$ the unique reduced walk in $T$ from $\verti$ to $\vertii$.
\end{nota}

At the core of the alluded unicity is the following well-known fact which we state without a proof.

\begin{prop}
\label{prop: generators from tree}
Let $\Ati$ be a connected $A$-automaton and let $T$ be a spanning tree of $\Ati$. Then
the set
\begin{equation*} \label{eq: generating set}
S_T
\,=\,
\big{\{}
\,
\red{\lab}(\bp \xleadsto{\scriptscriptstyle{T}} \! \bullet \! \arc{\edgi} \!\bullet\! \xleadsto{\scriptscriptstyle{T}} \bp)\st \edgi\in E^+\Ati \setmin ET \,
\big{\}}
\,\subseteq\, \Free[A]
\end{equation*}
is a generating set for
$\gen{\Ati}$. Furthermore, if $\Ati$ is reduced then
$S_T$ is a free basis for~$\gen{\Ati}$. \qed
\end{prop}

On the other hand, given a reduced word $w
\in \Free[A]$, we can always consider the
\defin{petal automaton}
$\flower{w}$, \ie
the cyclic $A$-automaton spelling $w$ (or $w^{-1}$ if read in the opposite direction).
Then, given a subset $S
=\set{w_i}_i
\subseteq \Free[A]$, we define the \defin{flower automaton} of~$S$, denoted by~$\flower{S}$, to be the automaton obtained after identifying the basepoints of the petals of the elements in $S$; that is,
$\flower{S} = \sum_i \flower{w_i}$.
\vspace{-10pt}
\begin{figure}[H]
\centering
\begin{tikzpicture}[shorten >=1pt, node distance=2cm and 2cm, on grid,auto,>=stealth',
decoration={snake, segment length=2mm, amplitude=0.5mm,post length=1.5mm}]
  \node[state,accepting] (1) {};
  \path[->,thick]
        (1) edge[loop,out=160,in=200,looseness=8,min distance=25mm,snake it]
            node[left=0.2] {$w_1$}
            (1);
            (1);
  \path[->,thick]
        (1) edge[loop,out=140,in=100,looseness=8,min distance=25mm,snake it]
            node[above=0.15] {$w_2$}
            (1);
  \path[->,thick]
        (1) edge[loop,out=20,in=-20,looseness=8,min distance=25mm,snake it]
            node[right=0.2] {$w_p$}
            (1);
\foreach \n [count=\count from 0] in {1,...,3}{
      \node[dot] (1\n) at ($(1)+(45+\count*15:0.75cm)$) {};}
\end{tikzpicture}
\vspace{-20pt}
\caption{The flower automaton $\flower{w_1,w_2,\ldots,w_p}$}
\label{fig:  flower automaton}
\end{figure}
It is clear that
$\gen{\flower{S}} = \gen{S} \leqslant \Free[A]$.
Hence, every subgroup $H$ of $\Free[A]$ is recognized by some
(clearly not unique)
$A$-automaton.
Note however that, if the words in $S$ are reduced, then
$\flower{S}$ is core, and deterministic
\emph{except maybe at the basepoint}.
A key result due to
J. R. Stallings
is that determinism is the only missing condition in order to make this representation unique.

\begin{defi} \label{def: Schreier automaton}
The
(right)
\defin{Schreier $A$-automaton} of a subgroup $H \leqslant G = \gen{A}$,
denoted by $\schreier{H}$,
is the automaton with vertices the (right)
cosets $Hg$ $(g \in G)$, arcs
\smash{$Hg \xarc{a\,} Hga$}
for every $a \in A^{\pm}$, and basepoint $H$.
Note that $\schreier{H}$ is always deterministic (but not necessarily core).
\end{defi}

\begin{defi} \label{def: Stallings automaton}
The \defin{Stallings $A$-automaton} of a subgroup
$H$ of $\Free[A]$
is the core of the
right
Schreier automaton
$\Schreier{H}$.
We denote it by $\Stallings{H,A}$
(or simply by $\Stallings{H}$ if the generating set is clear).
\end{defi}

Note that $\Stallings{\Free_A,A}$ has one single vertex and arcs labelled by each $a \in A$. Such an automaton is called a \emph{bouquet}. We denote by $\bouquet{n}$ a bouquet with $n$ positive arcs.

By construction, $\stallings{\Sgpi,A}$ is deterministic, core (\ie it is a reduced $A$-automaton) and recognizes $H$. Below, we see that every reduced $A$-automaton is the Stallings automaton of some subgroup of $\Fn$, making ``reduced'' and ``Stallings'' automata equivalent notions.

\begin{thm}[\cite{StallingsTopologyFiniteGraphs1983}]\label{thm: Stallings bijection}
Let $\Fn$ be the free group on $A = \{a_1, \ldots ,a_n\}$; then the map
 \begin{equation}\label{eq: Stallings bijection}
\begin{array}{rcl}
\Set{\text{Subgroups of } \Fn} & \to & \Set{\text{Reduced $A$-automata}} \\
H \ \ & \mapsto & \ \Stallings{H} \\
\gen{\Ati} \  & \mapsfrom &  \  \Ati
\end{array}
 \end{equation}
is a bijection. Furthermore, finitely generated subgroups correspond precisely to finite Stallings automata, and in this case the bijection is algorithmic. \qed
\end{thm}

As stated, if we restrict our attention to finitely generated subgroups,
then the above bijection is algorithmic.
Given a finite Stallings $A$-automata~$\Ati$, one can always compute a maximal tree $T$ of $\Ati$,
and then use
Proposition \ref{prop: generators from tree}
to compute a free basis for $\gen{\Ati}$.
In particular,
$\rk \gen{\Ati} =
{\card \Edgs^+\Ati} - \card \Verts\Ati+ 1 = \rk \Ati$.

For the other direction, suppose that we are given a finite set of generators
$S
$ for~$H$.
We have seen that $\flower{S}$
is a finite core $A$-automaton recognizing $H$, although it might be not deterministic.
To fix this, one can apply successive foldings on
possible arcs breaking determinism. Of course, a folding can
provide new opportunities for folding, but, since the number of arcs in the graph is finite and decreases with each folding, after a
finite number of steps, we will obtain an $A$-automaton with no available foldings (\ie deterministic, and hence a reduced $A$-automaton).
Theorem~\ref{thm: Stallings bijection} states that the resulting automaton must be precisely $\Stallings{H,A}$.
Note that the bijectivity of
\eqref{eq: Stallings bijection}
implies that the result of the folding process
depends neither on the order in which we perform the foldings
nor on the starting (finite) set of generators for $H$, but only on the subgroup~$H$ itself.

The bijection~\eqref{eq: Stallings bijection} has proven to be extremely fruitful and has provided natural proofs for many results on the free group,
especially from the algorithmic point of view (see~\cite{
KapovichStallingsFoldingsSubgroups2002,
MiasnikovAlgebraicExtensionsFree2007,
BartholdiRationalSubsetsGroups2010}).
In particular,
the Nielsen-Schreier theorem,
the solvability of the subgroup membership problem,
and the computability of rank and basis for finitely generated subgroups can be easily derived from this geometric interpretation.
Concretely, given a finite subset $S$ of $\Fn$, one can decide whether a given element $w \in \Fn$ belongs to $\gen{S}$ just by checking whether $w$ reads a \bp-walk in $\Stallings{S}$. In the same vein, one can use Theorem~\ref{thm: Stallings bijection} together with Proposition~\ref{prop: generators from tree} to compute a basis, and hence the rank, of the finitely generated subgroup~$\gen{S}$.

Many other algebraic properties of subgroups become transparent from this geometric viewpoint.
For example the Stallings automata of the conjugacy classes of a subgroup $H$ of $\Fn$ correspond exactly to the automata obtained after changing the basepoint in $\schreier{H}$ and taking the corresponding core automaton.

\begin{rem} \label{rem: fi iff finite and complete}
Note that $\Stallings{H}$ is complete if and only if $\Stallings{H} = \Schreier{H}$, and otherwise $\Schreier{H}$ is necessarily infinite. Hence a subgroup $H$ has finite index in $\Fn$ if and only if $\Stallings{H}$ is finite and complete.
\end{rem}
From this,
one can easily decide finite index in $\Fn$, and obtain other classical results involving subgroups of finite index, such as the Schreier index formula
for finite index subgroups
($\rk H = {\ind{\Fn}{H}} (n-1) + 1$).

With some extra work, one can also study intersections and extensions of subgroups of~$\Fn$
(we say that $G$ is an extension of $H$ if $H$ is a subgroup of $G$).
For example, a neat proof for an algorithmic version of the classical theorem of Takahasi on free extensions is easily obtained using Stallings automata.

\begin{thm}[\cite{TakahasiNoteChainConditions1951}] \label{thm: Takahasi theorem}
Every extension of a given finitely generated subgroup $H$ of $\Fn$ is a free multiple of an element of a computable finite family of extensions of $H$. \qed
\end{thm}
The minimal such computable set of extensions is called the set of
\defin{algebraic extensions}
of~$H$, and is denoted by~$\Alg(H)$.
Note that $H \in \Alg(H)$
(see~\cite{MiasnikovAlgebraicExtensionsFree2007} for details).

Finally, we extend the previous scheme to arbitrary
$A$-automata,
that is, we define the
\defin{Stallings reduction}
of an $A$-automaton $\Ati$ to be~$\stallings{\Ati} = \stallings{\gen{\Ati}}$. We usually abbreviate $\stallings{\Ati}$ to $\red{\Ati}$.
Note that we then have that
$\red{\Ati} = \red{\Atii} $ if and only if $ \gen{\Ati} = \gen{\Atii}$.

\begin{lem} \label{lem: red lang = red lang red}
If $\Ati$ is an involutive $A$-automaton, then $\rlang{\Ati} = \rlang{\red{\Ati}}$.
\end{lem}

\begin{proof}
It suffices to consider one folding at the time. More precisely, let $\Ati'$ be obtained from~$\Ati$ by folding the arcs
\smash{$\verti \xarc{a\,} \vertii$}
and
\smash{$\verti \xarc{a\,} \vertii'$}. Clearly, $\lang{\Ati} \subseteq \lang{\Ati'}$. On the other hand, for every $u \in \lang{\Ati'}$, we can find some $u' \in \lang{\Ati}$ by inserting factors of the form $a^{-1}a$ into $u$. It follows that $\overline{u'} = \overline{u}$ and so
$\rlang{\Ati} =
\red{\lang{\Ati}} = \red{\lang{\Ati'}} =
\rlang{\Ati'}$.
By iterating this argument for a sequence of foldings, we get the desired claim.
\end{proof}

\medskip

In the remainder of the section we summarize several consequences
(on the relation between the graphical and the algebraic rank)
of applying the Stallings bijection \eqref{eq: Stallings bijection}
to previous results.

The first one is an immediate consequence of the fact that foldings never increase the rank of the affected automata.

\begin{rem}
    Let $\Ati$ be an automaton recognizing $H \leqslant \Fn$; then,
    $
        \rk H
        =
        \rk \red{\Ati}
        \leq
        \rk \Ati
    $.
\end{rem}
In particular (of course),
if $H,K \leqslant \Fn$,
then~$\rk (H \join K) \leq \rk H + \rk K$, but this bound is not necessarily tight. Below,
we use Stallings theory to precisely describe the maximum and minimum ranks attainable by extending a given finitely generated subgroup of $\Fn$.

\begin{prop}
Let $H$ be a finitely generated subgroup of $\Fn$. Then,
\begin{enumerate}[ind]
    \item \label{item: min extended rank}
    the minimum rank of an extension of $H$ is the minimum of the ranks of the algebraic extensions of $H$.

\item \label{item: max extended rank}
the maximum rank of an extension of $H$ is either infinite (if $H$ is of infinite index in~$\Fn$) or equal to the rank of~$H$ (if $H$ is of finite index in~$\Fn$).
\end{enumerate}
In other words, for every $K \leqslant \Fn$,
\begin{equation*}
    \min_{H_i \in \Alg(H)}
    \rk (H_i)
    \,\leq\,
    \rk (H \join K)
    \,\leq\,
    \left\{ \!
    \begin{array}{ll}
         \infty & \text{if } \ind{\Fn}{H} = \infty  \\
         \rk (H) & \text{if } \ind{\Fn}{H} < \infty \, ,
    \end{array}
    \right.
\end{equation*}
and all the bounds are tight and computable.
\end{prop}

\begin{proof}
\ref{item: min extended rank}
Let $L$ be an extension of $H$ of minimum rank.
Since the algebraic extensions of $H$ are certainly extensions of $H$, it is clear that
$\rk(L) \leq
\min\, \set{\rk(H_i) \st H_i \in \Alg(H) }
$.
On the other hand, since every extension of $H$ is a free factor of some element in  $\Alg(H)$ (see Theorem~\ref{thm: Takahasi theorem}),
from Grushko theorem
we have that
$\rk(L) \geq \min\, \set{\rk(H_i) \st H_i \in \Alg(H) }$.
Hence, $\rk(L) = \min\, \set{\rk(H_i) \st H_i \in \Alg(H) }$, as claimed.

\ref{item: max extended rank}
If the index of $H$ in $\Fn$ is infinite, then
there exists a vertex $\verti$ in $\stallings{H}$ and a generator $a \in A$ such that there is no $a$-arc leaving $\verti$. So, it is enough to attach to $\verti$
an $a$-source of infinite rank (for example the one in~Figure~\ref{fig: a-source infinite rank})
to obtain an extension (indeed a free multiple) of $H$ of infinite rank.

\begin{figure}[H]
  \centering
  \begin{tikzpicture}[shorten >=1pt, node distance=1.2 and 1.2, on grid,auto,>=stealth']
   \node[state,accepting] (0) {};
   \node[state] (1) [right = of 0]{};
   \node[state] (2) [right = of 1]{};
   \node[state] (3) [right = of 2]{};
   \node[state] (4) [right = of 3]{};
   \node[state] (5) [right = of 4]{};
   \node[] (dots) [right = 0.5 of 5]{$\cdots$};

   \path[->]
        (0) edge[red]
            node[below] {\scriptsize{$\R{a}$}}
            (1);

    \path[->]
        (1) edge[loop above,blue,min distance=10mm,in=55,out=125]
            node[left = 0.1] {\scriptsize{$b$}}
            (1)
            edge[red]
            (2);

    \path[->]
        (2) edge[loop above,blue,min distance=10mm,in=55,out=125]
            (2)
            edge[red]
            (3);

    \path[->]
        (3) edge[loop above,blue,min distance=10mm,in=55,out=125]
            (3)
            edge[red]
            (4);
    \path[->]
        (4) edge[loop above,blue,min distance=10mm,in=55,out=125]
            (4)
            edge[red]
            (5);

    \path[->]
        (5) edge[loop above,blue,min distance=10mm,in=55,out=125]
            (5);
\end{tikzpicture}
\caption{An $a$-source of infinite rank}
\label{fig: a-source infinite rank}
\end{figure}

On the other hand, if $\ind{\Fn}{H} < \infty$ then $\stallings{H}$ is complete, and therefore any addition to $\stallings{H}$ would produce an identification of vertices. Therefore, after folding, we get a decrease in the index (and hence in the rank) of~$H$.
\end{proof}

If $R \subseteq \Fn$, then we usually abuse language and write  $\Ati + R = \Ati + \flower{R}$. Note that we then have that
$\gen{\Ati + R} =
\gen{\Ati} \join \gen{R}$.

\begin{lem} \label{lem: diff of ranks}
Let $\Atii$ and $\Ati$ be automata.
If $\Atii \leqslant \Ati$, then $\rk \Atii \leqslant \rk \Ati$.
Moreover, if $\Atii,\Ati$ are finite and connected then
there exists a subset $R$ of $\Fn$ of cardinal $\rk \Ati - \rk \Atii$ such that $\red{\Ati} = \red{\Atii + R}$.
\end{lem}

\begin{proof}
The first claim is obvious if $\Ati$ has infinite rank, and otherwise it follows from a straightforward analysis on \eqref{eq: graph rank} corresponding to the possible situations after removing an edge or an isolated vertex from~$\Ati$.
For the second claim,
it is enough to consider any spanning tree $T$ of the subgraph $\Atii$ and realize that it can be expanded
(\eg by using breadth-first search)
to a spanning tree $T'$ of $\Ati$.
We can then take $R$ to be $S_{T'} \setmin S_{T}$,
which clearly has the required cardinal.
Since $\gen{\Ati} = \gen{S_{T'}} = \gen{S_T \cup R} = \gen{\Atii + R}$, the claimed result follows from Theorem~\ref{thm: Stallings bijection}.
\end{proof}

Note that, in general, a strict subgraph $\Atii < \Ati$ can still have the same rank as $\Ati$
(\eg if~$\Ati$ has isolated vertices, or ``hanging trees''). Below we prove that this is no longer true if $\Ati$ is finite and reduced.

\begin{cor} \label{cor: strict inclusion}
The rank of a finite core automaton is strictly greater than the rank of any of its connected strict subautomata.
\end{cor}

\begin{proof}
Let $\Ati$ be a finite core automaton. Note that any connected subautomaton of~$\Ati$ is obtained by successively removing arcs, and discarding the eventual connected components not containing the basepoint that may appear. Since $\Ati$ is core, the first removed arc cannot produce an isolated vertex; hence the rank decreases in view of Lemma~\ref{lem: diff of ranks}.
\end{proof}

\begin{lem} \label{lem: identification first}
    Let $\Ati$ be a finite Stallings automaton, and
    let $R \subseteq \Fn$ be a finite subset
    of size~$r$.
    If
    $\card {\Verts (\red{\Ati + R})} < \card {\Verts \Ati}$,
    then
    $\red{\Ati + R} = \red{\Ati_{\!\!/\verti \shorteq \vertii} + R'}$,
    where
    $\verti,\vertii$
    are two different vertices in~$\Ati$, and $R' \subseteq \Fn$ is a finite subset of size at most $r - 1$.
\end{lem}

\begin{proof}
    First note that the hypothesis
    $\card {\Verts (\red{\Ati + R})} < \card {\Verts \Ati}$
    entails the identification of two (different) vertices in $\Ati$ during the reduction of $\Ati + R$.

    Let us call \defin{$\Ati$-free} any folding that does not produce an identification of
    vertices in $\Ati$.
    We then let $\Ati'$ be an automaton (note that it may not be uniquely determined) obtained after
    successively performing $\Ati$-free foldings on $\Ati + R$ until no more $\Ati$-free foldings are possible
    (in particular, $\Ati \leqslant \Ati'$ and $\red{\Ati'} = \red{\Ati + R}$).
    Note that $\Ati + R$ is core and so is $\Ati'$. But $\Ati'$ cannot be reduced
    (otherwise~$\red{\Ati + R} = \Ati' \geqslant \Ati$, contradicting the hypothesis that $\card {\Verts (\red{\Ati + R})} < \card {\Verts \Ati}$)
    and therefore strictly contains $\Ati$ as a subautomaton.
    Hence, from Corollary~\ref{cor: strict inclusion} and Lemma~\ref{lem: diff of ranks}, we have that  $\rk \Ati + 1 \leq \rk \Ati' \leq \rk \Ati +r$.

    Since $\Ati'$ is not reduced
    and there are no $\Ati$-free foldings remaining,
    there must be an available folding in $\Ati'$ identifying two different vertices
    (say $\verti , \vertii$)
    in~$\Ati$.
    Note that
    (since $\Ati$ is reduced by hypothesis) at least one of the arcs involved in the folding
    must lie outside~$\Ati$.
    Thus
    the following conditions hold:
    \begin{enumerate}[ind]


        \item $ \Ati_{\!\!/\verti \shorteq \vertii}$ is a subautomaton of the automaton $\Ati''$ obtained after performing the folding in $\Ati'$.

        \item $\rk \Ati' = \rk \Ati''$.
        (This follows immediately from \eqref{eq: graph rank},  because we lose exactly one vertex and one positive arc in $\Ati'$ after the folding, keeping the number of connected components equal to $1$.)
    \end{enumerate}
    Therefore:
    \begin{equation} \label{eq: rank - 1}
    \begin{aligned}
    \rk \Ati'' - \rk \Ati_{\!\!/\verti \shorteq \vertii}
    &\,=\,
    \rk \Ati' - \rk \Ati_{\!\!/\verti \shorteq \vertii}\\
    &\,=\,
    \rk \Ati' - \rk \Ati - 1\\
    &\,\leq\,
    \rk \Ati + r - \rk \Ati - 1\\
    &\,=\, r - 1 \,,
    \end{aligned}
    \end{equation}
    where we have used that $\rk \Ati' \leq \rk \Ati +r$.

    Finally, since $\Ati_{\!\!/\verti \shorteq \vertii} \leqslant \Ati''$, it follows from \eqref{eq: rank - 1} and Lemma~\ref{lem: diff of ranks} that there exists a subset $R'\subseteq \Fn$ of size at most $r-1$ such that
    $
    \red{\Ati + R}
    = \red{\Ati''} = \red{\Ati_{\!\!/\verti \shorteq \vertii} + R'}
    $,
    which is what we wanted to prove.
\end{proof}

\section{Complements and coranks of subgroups} \label{sec: complements and coranks}

Given a bounded lattice $(L,\join,\meet)$ with maximum $1$ and minimum $0$, we say that two elements $a,b \in L$ are \defin{$\join$-complementary} (\resp $\meet$-complementary)  if $a \join b = 1$ (\resp $a \meet b = 0$); then we also say that each of the elements is a \defin{$\join$-complement} (\resp \defin{$\meet$-complement}) of the other.
Two elements $a,b \in L$ are said to be
\defin{directly complementary}
if they are both \mbox{$\join$-complementary}
and \mbox{$\meet$-complementary}; that is
if $a \join b = 1$ and $a \meet b =0$;
then we also say that each is a
\defin{direct complement}
of the other.

Let $\SGP(G)$ denote the set of subgroups of a group $G$, and let $H,K \in \SGP(G)$.
We define the \defin{join}
of $H$ and $K$
to be
$H \join K = \gen{H \cup K}$.
It is easy to see that $(\SGP(G),\join,\cap)$ is a bounded lattice with maximum $G$ and minimum the trivial subgroup $\Trivial$.
This immediately provides the corresponding notions of complement in this setting.
If $H \cap K = \Trivial$, then we
say that $H$ and $K$ are
in \defin{direct sum},
and we write
$H \join K = H \oplus K $.
In particular, $H,K \leqslant G$ are directly complementary if and only if $H \oplus K = G$.

\begin{rem} \label{rem: free => oplus => join}
 Note that, if we denote the free product of two subgroups by $H * K$, then $H*K = G \Imp H \oplus K = G \Imp H \join K = G$, but both converses  are false. So, we have three natural increasingly restrictive ways of ``adding'' subgroups; namely joins ($\join$), direct sums ($\oplus$), and free products $(*)$. The corresponding notions of complement are summarized below.
 \end{rem}

 \begin{defi} \label{defi: subgroup complements}
Let $H,K \leqslant G$.
If $H \join K = G$
(\resp $H \oplus K = G$, $H * K = G$)
then we say that
$H$ and $K$ are~$\join$-\defin{complementary}
(\resp $\oplus$-\defin{complementary}, $*$-\defin{complementary})
in $G$, and that each of them is a
\defin{$\join$-complement}
(\defin{$\oplus$-complement}, \defin{$*$-complement})
of the other.
\end{defi}

 \begin{rem} \label{rem: conjug}
 We note that all three notions of complement are well behaved with respect to conjugation. More precisely, if $H,K \leqslant G$, then, for every $g \in G$:
 \begin{equation*}
     (H \join K)^{g} = H^{g} \join K^{g},
     \quad
     (H \oplus K)^{g} = H^{g} \oplus K^{g},
     \quad \text{and}
     \quad
     (H * K)^{g} = H^{g} * K^{g},
 \end{equation*}
 understanding that, in every equation, each of the sides is well defined if and only if the other side is well defined.
\end{rem}

We are interested in the behaviour of the complements of finitely generated subgroups within the free group. Concretely, we investigate which subgroups of $\Fn$ admit each kind of complement, and the
properties
of the respective sets of complements.
We note that, from the last remark, these notions work modulo conjugation.
The case of free products has been extensively studied and is well understood.
Namely, from Grushko Theorem, a free complement of a subgroup $H$ of $G$  necessarily has
(minimum and maximum)
rank equal to $\rk (G) - \rk (H)$. Moreover, given a finite subset $S$ of $\Fn$, one can decide whether $\gen{S}$ admits a free complement either using the classical Whitehead's peak reduction argument (see~\cite{WhiteheadEquivalentSetsElements1936}) or,
more efficiently,
using techniques based on the Stallings description of subgroups
(see~\cite{RoigComplexityWhiteheadMinimization2007,
SilvaAlgorithmDecideWhether2008,
PuderPrimitiveWordsFree2014}).
We aim to extend some of these results to our weakened notions of complement.

\begin{defi} \label{def: coranks}
Let $H$ be a subgroup of $G$; then
the
\defin{$\join$-corank}
(\resp \defin{$\oplus$-corank})
of $H$
is the minimum
rank of a $\join$-complement
(\resp $\oplus$-complement)
of $H$.
These are denoted by
$\jcrk(H)$,
and
$\dcrk(H)$
respectively.
\end{defi}

\begin{rem}
Note that
$\jcrk(H)
\leq
\dcrk(H)$,
whenever they are well defined.
Moreover,
if $H$ is a free factor of $G$,
then $\jcrk(H) = \dcrk(H) = \rk(G) - \rk(H)$,
which immediately follows from Grushko theorem.
In particular $\jcrk(\trivial) = \dcrk(\trivial) = \rk(G)$.
\end{rem}

A natural approach to this kind of questions is through the simplest possible complements, namely cyclic complements. We use the term cocycle to refer to cyclic complements (or their generators). Below is the precise definition in terms of each kind of complement.

\begin{defi}
 A \defin{$\join$-cocycle}
 (\resp \defin{$\cap$-cocycle}, \defin{$\oplus$-cocycle})
 of a subgroup  $H \leqslant \Fn$ is a cyclic $\join$-complement
 (\resp $\cap$-complement, $\oplus$-complement)
 of $H$, or any of its generators. We denote the corresponding sets of cocyles by
 $\Coc_{\join} (H)$, $\Coc_{\cap} (H)$, and $\Coc_{\oplus} (H)$
 respectively.
 A subgroup is said to be \defin{$\join$-cocyclic} (\resp \defin{$\oplus$-cocyclic}) if it admits a cyclic complement of the corresponding kind.
\end{defi}

Note that there is no need for the concept of $\cap$-cocyclic subgroup since admitting a cyclic $\cap$-complement is equivalent to admitting a (general) $\cap$-complement.

\section{Join complements} \label{sec: join}

Recall that a subgroup $K$ of $G$ is a $\join$-complement of $H \leqslant G$ (in $G$) if $H \join K = G$, and we define the $\join$-corank of $H$ to be the minimum possible rank for such a subgroup $K$.
We note that this concept has previously appeared in the literature under other names. Concretely, it appears in~\cite{PuderPrimitiveWordsFree2014} under the name of \emph{distance between subgroups}, where the author also  provides bounds for it and proves its computability.
We will present our own proofs here, obtained independently, for the sake of completeness.

Obviously, every subgroup $H$ of $\Fn$ admits a $\join$-complement of rank at most $n$
(namely~$\Fn$). It is easy to see that the same holds for proper $\join$-complements if the involved subgroup is not trivial.

\begin{lem}
Every nontrivial subgroup of $\Fn$ admits a proper $\join$-complement of rank $n$.
\end{lem}


\begin{proof}
    It is enough to prove the claim for nontrivial cyclic subgroups $\gen{u} \leqslant \Fn$, where $u$ is cyclically reduced.

    If~$n = 1$, then, given a nontrivial subgroup $\gen{a^k} \leqslant
    \ZZ =
    \pres{a}{-}$, it is enough to consider any proper subgroup $\gen{a^l} \leqslant \ZZ$, where $k,l$ are coprime integers.

    If $n\geq 2$, we distinguish two cases:
    \begin{enumerate}[dep]
        \item if the first and last letters of $u$ are equal
        (say
        to
        $a
        \in A^{\pm}
        $), then consider a subgroup of the form
        $ K =
        \gen{ \set{u a_i u^{-1} \st a_i \in A \setmin \set{a}}
        \cup \set{b a b^{-1}}}$, where $b \in A \setmin \set{a}$.

        \item if the first and last letters in $u$ are different (say equal to $b,a \in A^{\pm}$ respectively), then consider the subgroup
         $ K =
          \Gen{ \set{u a_i u^{-1} \st a_i \in A \setmin \set{a}}
         \cup \set{a}} $.
    \end{enumerate}
    Now, from the Stallings representation of $K$ it is clear that
    in both cases:
    $K \neq \Fn$, 
    $H \join K = \Fn$,
    and $\rk(K) = n$.
    We have that $K$ is a proper complement of $H$ of rank $n$,
    and the proof is concluded.
\end{proof}

Our next goal is to compute the $\join$-corank (\ie the minimum possible rank of a $\join$-complement) of a given finitely generated subgroup $H$ of $\Fn$.
Note that,
in the context of free groups,
we can restate this notion in graphical terms in the following way:
\begin{equation*}
    \jcrk(H)
    \,=\,
    \min \,
    \set{\,
    |S| \st
     S \subseteq \Fn \text{ and } \red{\stallings{H} + S} = \bouquet{n}
    \,} \,.
\end{equation*}

First, we describe the range of possible values for the $\join$-corank of a
subgroup of $\Fn$.

\begin{lem} \label{lem: join-corank bounds}
If  $H \leqslant \Free[n]$, then
\begin{equation} \label{eq: join-corank bounds}
    \max\, \set{0,n - \rk (H)}
    \,\leq\,
    \jcrk(H)
    \,\leq\,
    n\,.
\end{equation}
Moreover, for $n > 1$,
this is the only general restriction between the rank and the $\join$-corank of a subset of~$\Fn$. Namely,
for every pair $(r,c) \in [1,\infty] \times [1,n]$ satisfying $n - r \leq c \leq n$ there exists a subgroup of $\Fn$ with rank $r$ and $\join$-corank $c$.
\end{lem}

\begin{proof}
It is clear that $0 \leq \jcrk(H) \leq n$ since any basis of $\Fn$ has $n$ elements and is enough to generate $\Fn$. On the other hand, since one cannot generate $\Fn$ with fewer than $n$ elements, it is also clear that $\rk (H) + \jcrk(H) \geq n$. The condition in~\eqref{eq: join-corank bounds} follows.

Let $\Fn = \pres{a_1, a_2, \ldots, a_n}{-}$. For each $c \in [1,n]$ and
each positive $r \geq n-c$,
consider the subgroup $H(r,c)$ generated by
\begin{multline*}
\hspace{60pt}
\set{ a_{c+1}, \ldots , a_n }
 \\
 \,\cup\,
\set{\, (a_1a_2)^ia_1^2(a_2^{-1}a_1^{-1})^i : 2i+1 \in [1, r-n+c] \,}\\
\,\cup\,
\set{\, (a_1a_2)^ja_1a_2^2a_1^{-1}(a_2^{-1}a_1^{-1})^j : 2j+2 \in [2, r-n+c]\,} \,. \hspace{30pt}
\end{multline*}
Its Stallings automaton has loops at the basepoint labelled by $a_{c+1}, \ldots , a_n$ and a chain obtained by concatenating $r-n+c$ alternated cycles labelled by $a_1^2$ and $a_2^2$. For instance, if $n = 3$ then $\Stallings{H(4,2)}$ is
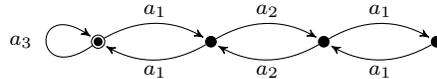
\vspace{-5pt}
\begin{figure}[H]
  \centering
  \begin{tikzpicture}[shorten >=1pt, node distance=1.5 and 1.5, on grid,auto,>=stealth']
   \node[state,accepting] (0) {};
   \node[state] (1) [right = of 0]{};
      \node[state] (2) [right = of 1]{};
   \node[state] (3) [right = of 2]{};

   \path[->]
        (0) edge[bend left]
            node[above] {\scriptsize{$a_1$}}
            (1)
        edge[loop left,min distance=10mm,in=145,out=215]
            node[left = 0.1] {\scriptsize{$a_3$}}
            (0);
    \path[->]
        (1) edge[bend left]
            node[below] {\scriptsize{$a_1$}}
            (0)
            (1) edge[bend left]
            node[above] {\scriptsize{$a_2$}}
            (2);
    \path[->]
        (2) edge[bend left]
            node[below] {\scriptsize{$a_2$}}
            (1)
        (2) edge[bend left]
            node[above] {\scriptsize{$a_1$}}
            (3);
    \path[->]
        (3) edge[bend left]
            node[below] {\scriptsize{$a_1$}}
            (2);
\end{tikzpicture}
\vspace{-15pt}
\end{figure}
It is clear that
$H(r,c)$ has rank $r$. To compute its $\join$-corank, we define the subgroup
\begin{equation*}
    K(r,c) = \gen{a_3,\ldots,a_c, a_1a_2, a_2(a_2a_1)^r}.
\end{equation*}
The given generating set is clearly a basis (every reduced word on the generators is actually reduced as a word on $A^{\pm}$); hence $K(r,c)$ has rank $c$.

Now $\Stallings{H(r,c)} + \Stallings{K(r,c)}$ has loops at the basepoint labelled by $a_3,\ldots,a_n$. After folding, we also get loops labelled by $a_1$ and $a_2$. Hence $\Stallings{H(r,c)} \join \Stallings{K(r,c)} = \Free[n]$ and so
$\jcrk(H(r,c))
    \,\leq\,
    c$. On the other hand, if we project $\Free[n]$ onto $(\ZZ / 2\ZZ)^n$, then most of the generators of $H(n,c)$ collapse and the image has rank $n-c$. Since $(\ZZ / 2\ZZ)^n$ has rank $n$, then it is impossible to find a $\join$-complement of $H(r,c)$ with rank strictly less than~$c$. Therefore $\jcrk(H(r,c)) = c$ and we are done.
\end{proof}

\begin{lem} \label{lem: inductivestep}
Let $\Ati$ be a nontrivial finite reduced automaton;
then, for every $r \geq 1$,
the join corank
$\jcrk\gen{\Ati} \leq r$
if and only if
$\jcrk \gen{\Ati_{\!\!/\verti \shorteq \vertii}} \leq r-1$, for some pair of distinct vertices $\verti,\vertii$ in~$\Ati$.
\end{lem}

\begin{proof}
    $[\Rightarrow]$
    If $\jcrk(\gen{\Ati}) \leq r$ then there exists a reduced automaton $\Atii$ of rank at most~$r$ such that~$\red{\Ati + \Atii} = \bouquet{n}$. Then, from Lemma~\ref{lem: identification first}, there exists a reduced automaton~$\Atii'$ of rank at most $r-1$ such that
    $\red{\Ati_{\!\!/\verti \shorteq \vertii} + \Atii'} =  \bouquet{n}$.
    Hence $\jcrk\gen{\Ati_{\!\!/\verti \shorteq \vertii}} \leq r-1$.

    $[\Leftarrow]$
    Since
   $\overline{\Ati_{\!\!/\verti \shorteq \vertii}} = \overline{\Ati + w}$, for some element $w\in\Fn$,
    and
    $\jcrk \gen{\Ati_{\!\!/\verti \shorteq \vertii}} \leq r-1$,
    there exists a reduced automaton $\Atii'$ of rank at most $r-1$ such that:
    \begin{equation*}
    \red{\Ati + (w + \Atii') }
    \,=\,
    \red{(\Ati + w) + \Atii'}
    \,=\,
    \red{\Ati_{\!\!/\verti \shorteq \vertii} + \Atii'}
    \,=\,
    \bouquet{n}.
    \end{equation*}
    So, $\gen{w + \Atii'}$ is a $\join$-complement of $\gen{\Ati}$ of rank at most $r$.
    Hence $\jcrk(\gen{\Ati}) \leq r$.
\end{proof}

\begin{thm}[\cite{PuderPrimitiveWordsFree2014}]
\label{thm: join-corank is computable}
    There exists an algorithm that, given a finite subset
    $S$ of $\Fn$, outputs the $\join$-corank of $\gen{S}$ in $\Fn$.
\end{thm}

\begin{proof}
    Let us write $\Ati = \stallings{S}$.
    It is clear that it is enough to be able to decide, for every $k \in [0,n-1]$, whether $\jcrk (\Ati) \leq k$.
    We proceed by induction on $k$. The case $k = 0$ being trivial, we assume that $k > 0$ and that $\jcrk (\Ati) \leq k-1$ is decidable.

    Suppose first that $\Ati$ is a bouquet (\ie it has a single vertex); then the $\join$-corank problem is trivial: $\jcrk{\Ati} = n-\rk(\Ati)$. Hence we may assume that $\Ati$ has at least two vertices.

    In view of Lemma~\ref{lem: inductivestep}, we have $\jcrk (\Ati) \leq k$
if and only if
$\jcrk \gen{\Ati_{\!\!/\verti \shorteq \vertii}} \leq k-1$, for some pair of different vertices $\verti,\vertii$ in~$\Ati$. By the induction hypothesis, we can check if this condition holds for every pair of vertices.
\end{proof}

\begin{cor}
    Let $H$ and $K$ be two
    finitely generated
    subgroups of $\Fn$
    (given by respective finite generating sets);
    then one can
    algorithmically
    decide whether $H\leqslant K$, and, if so, compute the $\join$-corank of $H$ in $K$.
\end{cor}

\begin{proof}
The first claim immediately follows from the solvability of the membership problem for free groups.
Namely, $H \leqslant K$ if and only if every one of the given generators for $H$ belongs to $K$, which can be checked by trying to read them as \mbox{\bp-closed} walks in $\Stallings{K}$.
In the affirmative case, the previous procedure provides the expression of the given $H$-generators as words in some
free basis
$\mathcal{B}$ of~$K$. Hence, the $\join$-corank of $H$ in $K$ is exactly the $\join$-corank of the
(subgroup generated by the)
new words in $\Free[\mathcal{B}]$, which is computable using Theorem~\ref{thm: join-corank is computable}.
\end{proof}

We show next that, for a fixed ambient $\Fn$ (and its canonical basis), we can compute the corank of a finitely generated subgroup $H$ of $\Fn$ in polynomial time with respect to the
size of $H$ (the number of vertices of $\Stallings{H}$).

\begin{cor}
Let $n \geq 1$ be fixed. There exists an algorithm which computes the $\join$-corank of a given $H \leq_\fg \Fn$ of size $m$ in time $O(m^{2(n+1)})$.
\end{cor}

\begin{proof}
We assume that $H$ is given through its Stallings automaton which, in turn, can be assumed to have more than one vertex.
Let $\Verts$ denote the vertex set of $\Ati = \Stallings{H}$, and let $m = \card \Verts$. Given $I \subseteq \Verts \times \Verts$, let $\Ati/I$ denote the involutive automaton obtained from $\Ati$ by identifying all pairs of vertices belonging to $I$. Let $n'$ be the number of letters from the canonical basis of $\Fn$ labelling edges in $\Stallings{H}$. In view of the proof of Theorem \ref{thm: join-corank is computable},
and folding being confluent,
in order to compute the corank of $H$ it would suffice to do the following:
\begin{itemize}
    \item to fold $\Stallings{H}/I$ for every $I \subseteq \Verts \times \Verts$ with $|I| \leq n'$, checking whether we get a bouquet;
    \item to register the smallest $r = |I|$ yielding a bouquet.
\end{itemize}
The corank of $H$ is then $n - n' + r$.

Thus it is enough to consider $O(m^{2n})$ subsets of $\Verts \times \Verts$.
Indeed, $|\Verts \times \Verts| = m^2$ and $\binom{m^2}{j} \leqslant m^{2j}$ for every $j \leqslant n$, yielding
\begin{equation*}
\textstyle{
\sum_{j=0}^n \binom{m^2}{j}
\,\leqslant\,
\sum_{j=0}^{n} m^{2j}
\,\leqslant\,
2m^{2n} .
}
\end{equation*}
By \cite[Theorem 1.6]{TouikanFastAlgorithmStallings2006}, we
know that we
can fold an involutive automaton with $v$ vertices and $e$ arcs in time $O(e + (v+e)\log^*(v))$, where $\log^*(v)$ is the smallest positive integer $k$ such that the $k$-fold iteration
of the base 2 logarithm
satisfies $\log_2^k(m) \leqslant 1$.

In the case of $\Stallings{H}/I$ we have $|\Verts| \leqslant m$, $|\Edgs| \leqslant 2mn$ and, although $m$ is a very coarse upper bound for $\log^*(m)$, it suffices to show that folding each $\Stallings{H}/I$
can be performed in time $O(m^2)$.
Therefore we can compute the $\join$-corank of $H$ in time $O(m^{2n}m^2) = O(m^{2(n+1)})$.
\end{proof}

\begin{lem} \label{lem: join-coc rational}
\label{pcor}
The set of $\join$-cocycles of a finitely generated subgroup $H$ of $\Fn$ is rational and we can effectively compute a finite automaton recognizing it.
\end{lem}

\begin{proof}
Let $\Fn = \pres{A}{-}$. Assume first that $\stallings{H}$ has a single vertex, that is, that $\stallings{H}$ is the bouquet $\stallings{H}$ for some~$B \subseteq A$.

We now assume that $\card{B} < \card{A} - 1$. It is immediate that we have no $\join$-cocycles. If $B = A$, then
$\Coc_{\join}(H) =  \Free[A]$ and we are also done.
Thus we may assume that $B = A \setminus \{ a \}$ for some $a \in A$. In this case, it is easy to see that
\begin{equation*}
\Coc_{\join}(H)
\,=\,
\Free[B] \, \set{ a,a^{-1}} \, \Free[B] \,,
\end{equation*}
where $\Red_B$ denotes the set of reduced words on the alphabet $B^{\pm}$. Hence $\Coc_{\join}(H)$ is rational whenever $\stallings{H}$ is a bouquet.

Thus we may assume that $\Ati = \stallings{H}$
has at least two vertices. Let $\Verti$ denote the set of all pairs of vertices $(\verti,\vertii) $ in $\Ati$ such that $\Ati/_{\verti = \vertii}$ is the bouquet ${\mathcal{B}}_n$ on $n$ letters. We claim that
\begin{equation} \label{pcor1}
    \Coc_{\join}(H)
    \,=\
    \bigcup_{
    \mathclap{(\verti,\vertii) \in \Verti}} \;
    (\lang[\bp\verti]{\Ati} \,
    \lang[\vertii\bp]{\Ati}) \cap \Red_A.
\end{equation}
Indeed, let $u \in \Red_A$. To compute the Stallings automaton of $H \join u$, we glue a cycle labeled by $u$ to the basepoint of $\Ati$ and fold. If some edge of this new cycle is not absorbed by $\Ati$ in the folding process, then $\Ati$ will be a subautomaton of $\stallings{H \join u}$. Since $\Ati$ has at least two vertices, then we won't get the desired bouquet. Hence we may assume that $u$ admits a reduced factorization $u = u_1u_2$ such that:
\begin{itemize}
    \item there exist paths $\bp \xleadsto{u_1} \verti$ and $\vertii \xleadsto{u_2} \bp$ in $\stallings{H}$;
    \item $\Ati/_{\verti = \vertii} = {\mathcal{B}}_n$.
\end{itemize}
Therefore (\ref{pcor1}) holds and we are done.
\end{proof}

\section{Meet complements} \label{sec: meet}

Again, it is obvious that every subgroup admits a $\cap$-complement (namely, the trivial subgroup).
In addition, if a subgroup admits a nontrivial  $\cap$-complement then it admits a nontrivial $\cap$-cocycle as well.


\begin{lem} \label{lem: incomplete => meet-compl}
If $\Stallings{H}$ is incomplete then $H$ admits a $\cap$-complement of any
rank.
\end{lem}

\begin{proof}
Let $\verti$ be a vertex in $\Stallings{H}$ with no outgoing arc labelled by (say) $a \in A$, and let $w \in \Fn$ be a reduced word reading a walk from the basepoint to $\verti$ in $\Stallings{H}$. Then the subgroup recognized by any automaton of the form \smash{$\bp \xleadsto{wa}\Delta$}, where $\Delta$ is a labelled digraph (\emph{of any rank}) sharing no vertex with the walk labelled by $w$,
has trivial intersection with~$H$.
\end{proof}

The characterization of the \emph{finitely generated}
proper $\cap$-complements in $\SGP(\Fn)$ follows from the previous result and a well-known general property of finite index subgroups.

\begin{prop} \label{prop: meet-complementable <=> inf index}
A finitely generated subgroup of $\Fn$ admits a nontrivial $\cap$-complement
(of any rank)
if and only if it has infinite index in $\Fn$.
Hence it is decidable whether a given finitely generated subgroup of $\Fn$ admits a nontrivial $\cap$-complement.
\end{prop}

\begin{proof}
$[\Rightarrow]$
If $H$ is a subgroup of finite index of $\Fn$ then the set $\set{H u^k \st k\in \NN}$ must be finite for every $u \in \Fn$. Therefore, for every $u \in \Fn$, there exists $k\geq 1$ such that $u^k \in H$. Since $\Fn$ is torsion-free, it follows that $H$ admits no nontrivial $\cap$-complement.

$[\Leftarrow]$ Since a subgroup $H$ of $\Fn$ has finite index if and only if its Stallings automaton is finite and complete
(see Remark~\ref{rem: fi iff finite and complete}),
the claim follows from Lemma~\ref{lem: incomplete => meet-compl}.
\end{proof}

\begin{lem} \label{lem: char cr meet-cocycles}
\label{nmcom}
Let $H \leqslant \Fn
$ be finitely generated,
let $m$ be the number of vertices of $\Ati = \stallings{H}$,
and let $u \in \Cred_{n} \setminus \{ 1\}$;
then, the following statements are equivalent:
\begin{enumerate}[dep]
    \item\label{item: nontrivial meet complement}
    $\gen{u} \cap H = \Trivial$
    \quad(\ie $\gen{u}$ is a (nontrivial) $\cap$-complement of $H$);
    \item\label{item: for all notin}
    $u^{k} \notin H$ for every $k \geq 1$;
        \item\label{item: factorial notin}
    $u^{m!} \notin H$;
    \item\label{item: for finitely many notin}
    $u^{k} \notin H$ for every $k = 1, \ldots,m$;
    \item\label{item: no m-walks}
    $u^m \notin \lang[\sbp \Verts\Ati]{\Ati}$
    \quad(\ie $u^m$ cannot be read from the basepoint within $\Ati$).
\end{enumerate}
\end{lem}

\begin{proof}
Note that
$\ref{item: nontrivial meet complement}
\Leftrightarrow
\ref{item: for all notin}$
by definition, whereas the implications
$
\ref{item: no m-walks}
 \Rightarrow
\ref{item: for all notin}
\Rightarrow
\ref{item: factorial notin}
\Rightarrow
\ref{item: for finitely many notin}
$ are clear
(and only the first one needs the assumption of $u$ being cyclically reduced).

Finally, to see that
$\ref{item: for finitely many notin}
\Rightarrow\ref{item: no m-walks}
$,
suppose (by contraposition) that
$u^m$ is readable in $\stallings{H}$ from the basepoint,
that is, there exists a walk
\begin{equation} \label{eq: repeated vertex walk}
    \bp = \verti_0  \xleadsto{u} \verti_1 \xleadsto{u} \cdots \xleadsto{u} \verti_m
\end{equation}
in $\stallings{H}$.
Since there are only $m$ vertices in $\stallings{H}$,
there must be a repetition among the vertices $\verti_0,\verti_1,\ldots,\verti_m$
in \eqref{eq: repeated vertex walk},
and, since $\stallings{H}$ is reduced,
the first repeated vertex --- say $\verti_k$, for some $k \in [1, m]$ --- must be the basepoint
(otherwise there would be two different walks reading $u$ arriving at the first repeated vertex). Hence there is a $\bp$-walk in $\stallings{H}$ reading $u^k$, that is, there exists some $k \in [1, m]$, such that $u^k \in H$, contrary to the condition in \ref{item: for finitely many notin}.
Therefore, $[\ref{item: for finitely many notin} \Rightarrow \ref{item: no m-walks}]$, and all the five statements are equivalent, as claimed.\end{proof}

The previous characterization provides a description of the set of cyclically reduced $\join$-cocycles as a regular subset of $\Fn$.

\begin{lem} \label{lem: cr cap-coc rational}
The set of cyclically reduced $\cap$-cocycles of a finitely generated subgroup $H$ of~$\Fn$ is a rational subset of $\Fn$,
and we can effectively compute a finite automaton recognizing it.
\end{lem}

\begin{proof}
Let
$\Ati = \stallings{H}$ and $m = \card{\Verts\Ati}$. Let
$X = \set{\bp} \times (\Verts\Ati)^{m}$ be the set of $(m+1)$-tuples of vertices in $\Ati$ starting at the basepoint;
then,
for every $\boldsymbol{\verti} = (\verti_0,\verti_1,\ldots,\verti_m) \in X$,
\begin{equation*}
K_{\boldsymbol{\verti}} \,=\,
\textstyle{
\bigcap\nolimits_{i=1}^{m}    \lang[\verti_{i-1}\verti_{i}]{\Ati}}
\end{equation*}
is the set of words which are readable between any of the successive vertices in~$\boldsymbol{\verti}$.
Note that $\bigcup_{\boldsymbol{\verti} \in X} K_{\boldsymbol{\verti}}$ is
exactly
the set of nontrivial words in $A^{\pm}$ whose $m$-th power is readable from the basepoint in $\Ati$.
Hence, according to Lemma~\ref{lem: char cr meet-cocycles}, the set of cyclically reduced $\cap$-cocycles of $H$ is precisely:
\begin{equation} \label{eq: cyclically reduced cocycles}
\Coc_{\cap}(H) \cap \Cred_{A}
\,=\,
    (\mathcal{C}_{A} \setmin
    \textstyle{
    \bigcup\nolimits_{\boldsymbol{\verti} \in X} K_{\boldsymbol{\verti}}})
    \cup \Trivial \,.
\end{equation}
Since the sets $\mathcal{C}_{A}$ and $K_{\boldsymbol{\verti}}$ are rational (for every $\boldsymbol{\verti}\in X$),
and $X$ is finite, the claimed result follows immediately from \eqref{eq: cyclically reduced cocycles} and the closure properties of rational languages.
\end{proof}



\section{Direct complements} \label{sec: direct}

Finally, we address the study of $\oplus$-complements of finitely generated subgroups of~$\Fn$.
First, we address the question of their existence (\ie the existence of $\cap$-complements that are also $\join$-complements).
Below we prove that the second requirement
does not suppose any real restriction
(\ie a finitely generated subgroup admits a $\oplus$-complement if and only if it admits a $\cap$-complement,
see Proposition~\ref{prop: meet-complementable <=> inf index}).

\begin{thm}
\label{complement}
Let $H$ be a finitely generated subgroup of $\Fn$ where $H \neq \Trivial \neq \Fn$;
then $H \leqslant \Fn$ admits a
nontrivial
$\oplus$-complement if and only if
it has infinite index in $\Fn$.
Hence it is decidable whether a given finitely generated subgroup of $\Fn$ admits a nontrivial $\oplus$-complement.
\end{thm}

\begin{proof}
$[\Rightarrow]$ Since $\oplus$-complements are, in particular, $\cap$-complements, this is immediate from Proposition~\ref{prop: meet-complementable <=> inf index}.

$[\Leftarrow]$
From the hypotheses, we may assume that $\stallings{H}$ is nontrivial, finite, and incomplete (say $a$-deficient, for some generator $a$ of $\Fn$).
Note also that since, for any $w \in \Fn$, $K$ is a $\oplus$-complement of $H$ if and only if $wKw^{-1}$ is a $\oplus$-complement of $wHw^{-1}$, we may replace $H$ by any conjugate at our convenience.

In particular, $H$ has a conjugate which arises from adjoining the arc
$\verti  \xarc{a\,} \bp$ to an \mbox{$a$-deficient} vertex $\verti$ of an inverse graph $\Atii$ having no vertices of degree $1$. We replace $H$ by this conjugate.

Now, (since $H$ is nontrivial) there must exist a walk $\verti \xleadsto{} \verti$ in $\Atii$ reading some nonempty reduced word~$u$. Also, since $\Stallings{H}$ is incomplete, we have $n > 1$, and hence we may fix some letter $b \in A \setminus \{ a \}$. We claim that
\begin{equation*}
  K
  =
  \gen{(A\setminus \{ a \})
  \cup
  \set{ a^{-1}uababa^{-1}ua}}
\end{equation*}
is a $\oplus$-complement of $H$.

Note that $a^{-1}ua$ is reduced because it reads a reduced walk in (the reduced automaton) $\stallings{H}$, and, since $b \neq a$, then $a^{-1}uababa^{-1}ua$ is also reduced. It follows that
\begin{equation*}
    \Stallings{K}
    =
    \flower{(A\setminus \{ a \})
    \cup
    \set{ a^{-1}uababa^{-1}ua}}.
\end{equation*}
Suppose that $w \in \Fn$ represents a nontrivial element in $H \cap K$.
Since $w$ labels a~\mbox{$\bp$-walk} in $\Stallings{K}$, it has as prefix one of the words in $(A\setminus \{ a \}) \cup \{ a^{-1}uababa^{-1}ua \}$ or its inverse. However none of these words can be read off the basepoint of $\Stallings{H}$. This is obvious for $(A\setminus \{ a \})^{\pm}$. Suppose now that we have a walk
\begin{equation*}
   \bp \xleadsto{a^{-1}ua} \vertii \xleadsto{baba^{-1}ua} \cdots
\end{equation*}
in $\Stallings{H}$. Since $\Stallings{H}$ is reduced, we necessarily have that $\vertii = \bp$. But then we would be unable to read $b$ from the basepoint, since the only arc leaving it has label~$a^{-1}$. Similarly, we show that $w$ cannot have $a^{-1} u^{-1} a b^{-1} a^{-1} b^{-1} a^{-1} u^{-1} a
$ as a prefix. Therefore $H \cap K = \{ 1 \}$.

On the other hand, we can obtain $\Stallings{H \join K}$ by identifying the basepoints of $\Stallings{H}$ and $\Stallings{K}$, followed by a complete folding. Note that every $c \in A\setminus \{a\}$ labels a loop at the basepoint of $\Stallings{H \join K}$ because this already happens in $\Stallings{K}$. Hence it suffices to show that $a$ also labels a loop at the basepoint of $\Stallings{H \join K}$.

Since $a^{-1}uababa^{-1}ua$ lies in $K$ and is reduced, it is accepted by $\Stallings{K}$, and therefore by $\Stallings{H \join K}$. Thus we have a walk
\begin{equation*}
\bp \xleadsto{a^{-1}ua} \vertii_1 \xarc{b} \vertii_2 \xarc{a} \vertii_3 \xarc{b} \vertii_4 \xleadsto{a^{-1}ua} \bp
\end{equation*}
in $\Stallings{H \join K}$.
Since $ a^{-1}ua $ is accepted by $\Stallings{H}$ and $\Stallings{H \join K}$ is inverse, we obtain that $\vertii_1 = \vertii_4 = \bp$. Similarly, it follows from $b$ being accepted by $\Stallings{K}$ that
$\vertii_2 = \vertii_3 = \bp$. Thus $a$ labels a loop at the basepoint of $\Stallings{H \join K}$.

Hence $H \join K = \Fn$, and so $K$ is a $\oplus$-complement of $H$ as claimed.
\end{proof}

The following result is an immediate consequence of the above proof.

\begin{cor} \label{complebound}
Every $\oplus$-complementable finitely generated subgroup of $\Fn$ admits a $\oplus$-com\-ple\-ment of any rank greater than or equal to $n$.\qed
\end{cor}

Suggestively enough, a kind of dual of the previous situation applies within the subgroups admitting $\oplus$-complements; namely, the requirement of trivial intersection does not affect the rank bounds of the possible complements
(\ie they coincide with those obtained for $\join$-complements in Lemma~\ref{lem: join-corank bounds}).

\begin{lem} \label{lem: direct-corank bounds}
Let  $H$ be a $\oplus$-complementable subgroup of $\Free[n]$, then
\begin{equation} \label{eq: direct-corank bounds}
    \max\, \set{0,n - \rk (H)}
    \,\leq\,
    \dcrk(H)
    \,\leq\,
    n\,.
\end{equation}
Moreover, for $n > 1$
this is the only general restriction between the rank and the $\oplus$-corank of a subset of~$\Fn$. Namely,
for every pair $(r,c) \in [1,\infty] \times [1,n]$ satisfying $n - r \leq c \leq n$, there exists a subgroup of $\Fn$ with rank $r$ and $\oplus$-corank $c$.
\end{lem}

\begin{proof}
Since $\jcrk{H} \leq \dcrk{H}$, the bounds in \eqref{eq: direct-corank bounds} are immediate from those in \eqref{eq: join-corank bounds} and Corollary \ref{complebound}.
For the second claim it is enough to consider again
the families of subgroups $H(r,c)$ and $K(r,c)$ introduced in the proof of Lemma \ref{lem: join-corank bounds} and check that $H(r,c) \cap K(r,c) = \{ 1 \}$.

Indeed, suppose that $u \in H(r,c) \cap K(r,c)$. If we write $u$ as a reduced word on the generators of $K(r,c)$, say $u = u_1\ldots u_m$, then the generators $a_3,\ldots,a_c$ must be absent because they don't label edges in $\Stallings{H(r,c)}$. Suppose that $u_i \in \{ a_2(a_2a_1)^r, (a_1^{-1}a_2^{-1})^ra_2^{-1} \}$ for some $i$. Since there is no cancellation between consecutive $u_j$, it follows that either $(a_2a_1)^r$ or $(a_1^{-1}a_2^{-1})^r$ is a factor of $u$. But neither of these words labels a path in $\Stallings{H(r,c)}$; hence $u = (a_1a_2)^m$ for some $m \in \mathbb{Z}$. The only such word accepted by $\Stallings{H(r,c)}$ is $u = 1$; thus $H(r,c) \cap K(r,c) = \{ 1 \}$ as required.
\end{proof}

So, whenever the direct corank exists, the restriction of having trivial intersection does not affect the range of values it can take. In Question~\ref{qu: direct-corank = join-corank} we ask whether the same thing happens for the coranks themselves.

\subsection{Direct cocycles}

Throughout this section, $H$ will denote a finitely generated subgroup of $\Fn$. Recall that
$C_n$ denotes the set of cyclically  reduced words in $\Fn$,
and the set of direct cocycles of $H$ is
$ 
\dCoc{H}
\,=\,
\Set{ u \in \Fn \st H \oplus \gen{u} = \Fn } .
$ 
Since
$\Coc_{\oplus}(H)
=
\Coc_{\join}(H) \cap \Coc_{\cap}(H)
$,
the next result follows immediately from Lemmas \ref{lem: join-coc rational} and \ref{lem: cr cap-coc rational}.

\begin{cor}
\label{crrat}
The set of cyclically reduced $\oplus$-cocycles of a finitely generated subgroup $H$ of $\Fn$
(namely, $\dCoc{H} \cap \Cred_n$)
is a rational subset of $\Fn$ and we can effectively compute a finite automaton recognizing it. \qed
\end{cor}

However, the following example shows that rationality is no longer ensured for the full set of  $\oplus$-cocycles.

\begin{exa}
\label{exa: direct-Coc not rational}
Let $H = \gen{ a,b^2 } \leqslant
\Free[\{a,b\}] =
\Free[2]
$; then $\Coc_{\oplus}(H)$ is not a rational subset of $\Free[2]$.
\end{exa}

\begin{proof}
By Benois Theorem \cite{BenoisPartiesRationnellesGroupe1969}, a subset $S$ of $\Free[\set{a,b}]$ is rational
if and only if
it constitutes a rational language of $\set{a,b}^{*}$. We show that
\begin{equation} \label{nonra1}
a^*b(a^{-1})^* \setminus \Coc_{\oplus}(H)
\,=\,
\{ a^nba^{-n} \st n \geq 0 \} \,.
\end{equation}
Let $n \geq 0$. Clearly, $(a^nba^{-n})^2 = a^nb^2a^{-n} \in H \cap \gen{a^nba^{-n}}$. It follows that $a^nba^{-n} \in a^*b(a^{-1})^* \setminus \Coc_{\oplus}(H)$.

Conversely, assume that $m, n \geq 0$ are distinct and that $a^mba^{-n} \notin \Coc_{\oplus}(H)$.
From the Stallings automaton of $H$:
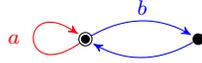
\vspace{-10pt}
\begin{figure}[H]
  \centering
  \begin{tikzpicture}[shorten >=1pt, node distance=1.5 and 1.5, on grid,auto,>=stealth']
   \node[state,accepting] (0) {};
   \node[state] (1) [right = of 0]{};

   \path[->]
        (0) edge[blue,bend left]
            node[above] {\scriptsize{$b$}}
            (1)
        edge[loop left,red,min distance=10mm,in=145,out=215]
            node[left = 0.1] {\scriptsize{$a$}}
            (0);

    \path[->]
        (1) edge[blue,bend left]
            (0);
\end{tikzpicture}
\vspace{-15pt}
\end{figure}
\noindent
it is clear that $H \join \gen{ a^mba^{-n} } = \Free[2]$ (since both extremes of the words $a^mba^{-n}$ would be collapsed by the loop reading $a$,
leaving a loop reading $b$ also attached to basepoint).
Hence $(a^mba^{-n})^k \in H$ for some $k \geq 1$.
But $\red{(a^mba^{-n})^k} = a^m(ba^{m-n})^{k-1}ba^{-n}$.
Since $m \neq n$, the reduced word $a^m(ba^{m-n})^{k-1}ba^{-n}$ is not accepted by $\stallings{H}$, a contradiction. Therefore $a^*b(a^{-1})^* \setminus \Coc_{\oplus}(H) \subseteq \{ a^nba^{-n} \st n \geq 0 \}$ and (\ref{nonra1}) holds.

Now $\{ a^nba^{-n} \st n \geq 0 \}$ is a classical example of a non-rational language (see \eg \cite[Exercise 4.1.1.c]{HopcroftIntroductionAutomataTheory2006}). Since rational languages are closed under boolean operations, it follows from (\ref{nonra1}) that $\Coc_{\oplus}(H)$ is not a rational subset of $\Free[2]$.
\end{proof}

We recall now the definition of a context-free language. This can be done through the concept of a context-free grammar, a particular type of rewriting system. Let $A$ be a finite alphabet. A {\em context-free $A$-grammar} is a triple ${\mathcal{G}} = (V,P,S)$ such that $V$ is a finite set disjoint from $A$, $S \in V$ and $P$ is a finite subset of $V \times (V \cup A)^{*}$. The language generated by~$\mathcal{G}$ is
\begin{equation*}
\lang{\mathcal{G}}
\,=\,
\set{\, w \in A^{\! *} \st S \overset{*\, }{\Rightarrow} w \,} \,,
\end{equation*}
that is the set of all words on $A$ obtained from $S$ by successively replacing a letter on the left-hand side of some element of $P$ by its right-hand side. A language $L \subseteq A^{\!*}$ is {\em context-free} if $L = \lang{\mathcal{G}}$ for some context-free $A$-grammar ${\mathcal{G}}$.
A subset $X$ of $\Fn$ is said to be context-free if the set of reduced forms of $X$ constitutes a context-free $A^{\pm}$-language.

\begin{prop}
\label{cfree}
Let $H$ be a finitely generated subgroup of $\Fn$; then $\Coc_{\oplus}(H)$ is a context-free subset of $\Fn$, and we can effectively compute a context-free grammar generating it.
\end{prop}

\begin{proof}
Let $\Ati = \stallings{H}$. It is clear that the claim holds whenever $\Ati$ is a bouquet; so, we may assume that $\Ati$ has at least two vertices.
Then, for every $\vertiii \in \Verts \Ati$, we denote by $H_\vertiii$ the conjugate $u^{-1}Hu$, where $u \in \lang[\bp \vertiii]{\Ati}$. It is easy to see that $\stallings{H_\vertiii}$ can be obtained from $\stallings{H}$ by making $\vertiii$ the new basepoint and successively erasing vertices  of degree $1$ distinct from~$\vertiii$. Let $K_\vertiii = \dCoc{H_\vertiii} \cap C_n$.
We show that:
\begin{equation} \label{eq: cfree1}
\dCoc{H}
\,=\,
\left(
\bigcup\nolimits_{\vertiii \in \Verts \Ati}
\Set{ vwv^{-1} \st v \in \lang[\bp \vertiii]{\Ati}, w \in K_\vertiii }
\right)
\,\cap\,
\Red_n.
\end{equation}
Let $vwv^{-1}$ belong to the right hand side as prescribed. Then $w \in \Coc(v^{-1}Hv) \cap C_n$ and so $v^{-1}Hv \oplus \gen{ w } = \Fn$, yielding $H \oplus \gen{ vwv^{-1} } = \Fn$. Thus $vwv^{-1} \in \dCoc{H}$.

Conversely, let $u \in \dCoc{H}$. As we saw in the proof of Lemma~\ref{lem: join-coc rational},
there exists
some $(\verti,\vertii) \in P$ and paths
$\vertii \xleadsto{\smash{u_2}} \bp \xleadsto{\smash{u_1}} \verti$
in $\stallings{H}$ such that $u = u_1u_2$. Write $u = vwv^{-1}$ with $w \in C_n$. Since $vwv^{-1} = u_1u_2$,
$v$ must be a prefix of $u_1$ or $u_2^{-1}$ (or both).
So we get $v \in \lang[\bp \vertiii]{\Ati}$ for some $\vertiii \in \Verts \Ati$. It remains to prove that $w \in K_\vertiii = \Coc(v^{-1}Hv) \cap C_n$; but $H \oplus \gen{ vwv^{-1} } = \Fn$ yields $v^{-1}Hv \oplus \gen{ w } = \Fn$ and so \eqref{eq: cfree1} holds.

Since the class of context-free languages is closed under finite union and intersection with rational languages (see, for example \cite[Section 10.5]{DavisComputabilityComplexityLanguages1994}), it suffices to show that the language $\{ vwv^{-1} \st v \in \lang[\bp \vertiii]{\Ati}, w \in K_\vertiii \}$ is context-free for every $\vertiii \in \Verts \Ati$.

Let $\mathcal{G} = (V,P,S)$ be the context-free
$(A^{\pm} \cup \set{\$})$-grammar defined by $V = \{ S\}$ and $P = \{ S \} \times (\{ \$\} \cup \{ aSa^{-1} \st a \in A^{\pm} \})$; then $\lang{\mathcal{G}} = \{ v\$ v^{-1} \st v \in (A^{\pm})^* \}$ is context-free. Since
$\lang[\bp \vertiii]{\Ati} \$ \lang[\bp \vertiii]{\Ati}^{-1}$ is rational, it follows that
$\lang{\mathcal{G}} \cap \lang[\bp \vertiii]{\Ati} \$ \lang[\bp \vertiii]{\Ati}^{-1}$ is context-free as well.

Define a homomorphism $\varphi$ from $(A^{\pm} \cup \{ \$\})^*$ into the monoid of rational $A^{\pm}$-languages (under product) defined by $a\varphi = a$ $(a \in A^{\pm})$ and $\$ \varphi = K_\vertiii$. This is an example of a {\em rational substitution}. Since
\begin{equation*}
\Set{ vwv^{-1} \st v \in \lang[\bp \vertiii]{\Ati}, w \in K_\vertiii }
\,=\,
(\lang{\mathcal{G}} \cap \lang[\bp \vertiii]{\Ati}\$ \lang[\bp \vertiii]{\Ati}^{-1})\varphi
\end{equation*}
and the class of context-free languages is closed under rational (in fact, context-free) substitution (see \cite[Theorem 4.1.1]{ShallitSecondCourseFormal2008}), we obtain the desired result.
\end{proof}

\subsection{The cocyclic case and a question}

A subgroup is said to be ($\join$ or $\oplus$)-\defin{cocyclic} if it admits a cyclic complement of the corresponding type.
Obviously, a $\oplus$-cocyclic subgroup is necessarily $\join$-cocyclic, but the converse fails for finite index proper subgroups in view of Theorem \ref{complement}.

\begin{thm}
Let $H$ be a finitely generated subgroup of infinite index of $\Fn$; then $H$ is $\oplus$-cocyclic if and only if it is $\join$-cocyclic.
\end{thm}

\begin{proof}
We prove the nontrivial converse implication.
Let $\Ati = \stallings{H}$ and let $\Verts = \Verts\Ati$.
Assume that $H \join \gen{u} = \Fn$. In view of Remark~\ref{rem: conjug}, we may assume
that every vertex of $\Ati$ has
degree strictly greater than
$1$. We may also assume that $\Ati \neq \bouquet{n-1}$, a trivial case.

Suppose that $u$ is not cyclically reduced. Write $u = vwv^{-1}$ as a reduced product with $w$ cyclically reduced.
If $v \notin \lang[\bp \Verts]{\Ati}$
then $\Ati$ remains a subautomaton of $\overline{\Ati +u} = \bouquet{n}$ after folding, a contradiction.
Thus there exists a walk
\smash{$\bp \xleadsto{v} \vertiii$} in $\Ati$. By replacing $H$ by its conjugate $H^{v}$
it follows that $H^{v} \join \gen{w} = \Fn$. Therefore, in view of Remark~\ref{rem: conjug}, we may assume that $u$ is cyclically reduced.

We may assume that $H \cap \gen{u} \neq \Trivial$. Hence $u^k \in H$ for some $k \geq 1$
which we can assume to be minimum,
and therefore there exists a walk in $\Ati$ of the form
\begin{equation*}
\bp = \vertii_0 \xleadsto{u} \vertii_1 \xleadsto{u} \ldots \xleadsto{u} \vertii_k = \bp
\end{equation*}
\textbf{Case 1:} Suppose that
$\lang[\vertii_{i-1} \Verts]{\Ati} \neq \lang[\vertii_i \Verts] {\Ati}$ for some $i$.
Replacing $u$ by $u^{-1}$ if necessary, we may assume that
$\lang[\vertii_{i-1}\Verts]{\Ati} \not\subseteq \lang[\vertii_i \Verts]{\Ati}$.
Replacing the basepoint by $\vertii_{i-1}$ through conjugacy if needed, we may assume that
$\lang[\vertii_{0} \Verts]{\Ati} \not\subseteq \lang[\vertii_1 \Verts] {\Ati}$.
We claim that
then
there exists some reduced $v' \in \lang[\vertii_{0}\Verts]{\Ati} \setmin \lang[\vertii_1\Verts]{\Ati}$.

Indeed, if
$v \in \lang[\vertii_{0} \Verts]{\Ati} \setmin \lang[\vertii_1 \Verts]{\Ati}$
is not itself reduced, we may write $v = xaa^{-1}y$ where $xa$ is reduced and $a \in A^{\pm}$.
Then, if
$xa \notin \lang[\vertii_1\Verts]{\Ati}$, we can take $v'$ to be $xa$.
Otherwise, we can replace $v$ by $xy \in \lang[\vertii_{0}\Verts] {\Ati} \setminus \lang[\vertii_1\Verts] {\Ati}$ (note that $xa, xy \in \lang[\vertii_1\Verts] {\Ati}$ would imply $v = xaa^{-1}y \in \lang[\vertii_1\Verts] {\Ati}$ since $\Ati$ is inverse).
Iterating this procedure, we end up finding some reduced
$v' \in \lang[\vertii_{0}\Verts] {\Ati} \setminus \lang[\vertii_1\Verts] {\Ati}$.
Therefore we may assume that $v$ is itself reduced.

Thus there exists a walk $\gamma \colon \bp \xleadsto{v} \vertii$ in $\Ati$ reading the (reduced) word $v$.
We claim that there exists some walk
\smash{$\vertii \xleadsto{w} \bp$}
such that $vw$ is reduced.
Indeed, since every vertex of $\Ati$ has outdegree $> 1$, we can extend
$\walki$ as much as needed keeping its label reduced. But since~$\Ati$ is finite, we must reach a point when one of the new vertices has already appeared before in
the walk. Let $\verti$ be the first such repeated vertex;
then we have paths
\begin{equation*}
\bp \xleadsto{x} \verti \xleadsto{y} \verti
\end{equation*}
where $xy$ is reduced
and has $v$ as a prefix.
Now
$xyx^{-1} \in \lang[q_{0}\Verts] {\Ati}
\setminus
\lang[q_1\Verts] {\Ati}$
and is reduced by the minimality of $\verti$. Writing $xyx^{-1} = vw$, we get the desired $w$.

We would like to have $vwu$ cyclically reduced, but that may fail. To overcome this difficulty, we will prove the claim below.
\begin{clm} \label{clm: claim}
If $a,b \in A^{\pm} \cap \lang[\bp \Verts] {\Ati}$ are distinct, then $\lang{\Ati} \cap \Cred_n \cap a(A^{\pm})^*b^{-1} \neq \emptyset$.
\end{clm}
Once again, we may extend the arc $\bp \xarc{a\,} \vertii$ to a walk $\bp \xleadsto{x} \verti$ with reduced label, and assume that $\verti$ is the first repeated vertex within the added arcs; then we use the same backtracking technique as before to get some $y \in \lang{\Ati} \cap \Red_n \cap a(A^{\pm})^*$. If $y$ ends in $b^{-1}$, we are done; otherwise, we may extend the arc
\smash{$\bp \xarc{b\,} \vertii'$} to a walk
\smash{$\bp \xleadsto{x'} \verti'$} with reduced label, and assume that $\verti'$ is the first repeated vertex within the added arcs. Again, the previous backtracking technique provides some
$y' \in \lang{\Ati} \cap \Red_n \cap b(A^{\pm})^*b^{-1}$.
But then
$yy' \in \lang{\Ati} \cap \Cred_n \cap a(A^{\pm})^*b^{-1}$. Therefore Claim~\ref{clm: claim} holds.

Since every vertex of $\Ati$ has degree strictly greater than $1$, it follows from Claim~\ref{clm: claim} that there exist cyclically reduced words $z_1, z_2 \in \lang{\Ati}$ such that $z_1vwz_2u \in \lang[\bp q_1]{\Ati} \cap \Cred_n$.
However, we cannot ensure that $z_1vwz_2u \notin \lang[q_1 \Verts]{\Ati}$. Therefore we consider $z_1^{m!}vwz_2u \in \lang[\bp q_1]{\Ati} \cap \Cred_n$, where $m = \card \Verts$. Suppose that $z_1^{m!}vwz_2u \in \lang[q_1 \Verts] {\Ati}$.
In view of Lemma~\ref{lem: char cr meet-cocycles}, we get that $vwz_2u \in \lang[q_1 \Verts]{\Ati}$, contradicting $v \notin \lang[q_1 \Verts]{\Ati}$. Thus we may replace $v$ (respectively $w$) by $z_1^{m!}v$ (respectively $wz_2$) to assume that $vwu$ is cyclically reduced.

Clearly, $H \join \gen{vwu} = H \join \gen{u} = \Fn$. Since $v \notin \lang[q_1 \Verts] {\Ati}$ and $vwu$ is cyclically reduced, we have that $(vwu)^{2} \notin \lang[\bp \Verts] {\Ati}$. Hence $H \cap \gen{vwu} = \{ 1 \}$ and so $H$ is $\oplus$-cocyclic in case 1.

\textbf{Case 2:} Suppose now that Case 1 does not hold, that is
$\lang[q_i \Verts] {\Ati} = \lang[q_j \Verts] {\Ati}$ for all $i,j \in [0,k]$.
We claim that this is incompatible with our assumptions.
Since
$\bouquet{n} =
\red{\Ati + u} =
\overline{\Ati_{\!\!/\bp \shorteq q_1}}$, then
in view of Lemma~\ref{lem: red lang = red lang red},
$\Red_n \subseteq \rlang{\Ati_{\!\!/\bp \shorteq q_1}}$.

Let $\Ati'$ be the automaton obtained from $\Ati$ by adding the arcs $\bp \xarc{1\,} q_1$ and $q_1 \xarc{1\,} \bp$
(we may admit arcs labelled by the empty word in automata with all the obvious adaptations).
It is straightforward to see that
$\lang{\Ati_{\!\!/\bp \shorteq q_1}}
=
\lang{\Ati'}$. We show that
$\lang[\bp \Verts] {\Ati'} \subseteq \lang[\bp \Verts] {\Ati}$
by induction on the number $t$
of arcs labelled by $1$ appearing in a walk
\smash{$\bp \xleadsto{v} $ in $\Ati'$}. Then $v \in \lang[\bp \Verts] {\Ati}$ holds trivially for $t = 0$. Assume that $t > 0$ and that the claim holds for $t-1$. Let
\smash{$\verti \xarc{1\,} \vertii$}
be the last such arc occurring in the walk; then, we can split our walk into
\begin{equation*}
\bp \xleadsto{x} \verti \xarc{1\,} \vertii \xleadsto{y} \vertiii
\end{equation*}
where $\verti,\vertii \in \{ \bp, q_1\}$ and are different. Then
$y \in \lang[\vertii \Verts] {\Ati}
=
\lang[\verti \Verts]{\Ati}$
and so there exists a walk
\smash{$\verti \xleadsto{y} \vertiii'$}
in $\Ati$. Gluing this walk to
\smash{$\bp \xleadsto{x} \verti$}, we obtain a walk labelled by $v$ which has $t-1$ occurrences of arcs labelled by 1. By the induction hypothesis, we obtain a walk
\smash{$\bp \xleadsto{v} $ in $\Ati$}.
Thus, $\lang[\bp \Verts] {\Ati'} \subseteq \lang[\bp \Verts] {\Ati}$.

From the previous discussion we have that:
\begin{equation*}
\Red_n
\,\subseteq\,
\rlang{\Ati_{\!\!/\bp \shorteq q_1}}
\,=\,
\rlang{\Ati'}
\,\subseteq\,
\rlang[\bp \Verts] {\Ati'}
\,\subseteq\,
\rlang[\bp \Verts] {\Gamma}
\,\subseteq\,
\lang[\bp \Verts] {\Gamma}
\end{equation*}
(where the last inclusion follows from $\Ati$ being inverse)
contradicting the fact that $\Ati$ is incomplete (since $H$ has infinite index). Therefore Case 2 is excluded and the theorem is proved.
\end{proof}

So, we have just seen that, whenever the $\oplus$-corank is defined, $\jcrk{H} =1$ if and only if $\dcrk{H} =1$. This fact, together with the coincidence in the possible
ranges
for both coranks (see Lemmas~\ref{lem: join-corank bounds} and~\ref{lem: direct-corank bounds}), makes it natural to ask whether the same holds for any possible value of the $\oplus$-corank.

\begin{qu} \label{qu: direct-corank = join-corank}
Do the $\oplus$-corank and the $\join$-corank of a finitely generated subgroup $H$ $\Fn$ coincide whenever they are both defined (\ie when $\ind{\Fn}{H} = \infty$)?
\end{qu}

\section*{Acknowledgements}

Both authors were partially supported by CMUP (UID/MAT/00144/2019), which is funded by FCT (Portugal) with national (MCTES) and European structural funds through the programs FEDER, under the partnership agreement PT2020.

 \bibliography{myBibtex}


\end{document}